\documentclass[12pt]{article}

\usepackage[top=1.5in, bottom=1.5in, left=1in, right=1in]{geometry}
\usepackage{latexsym}

\usepackage{amsmath}
\usepackage{amscd}
\usepackage{amsthm}
\usepackage{amsxtra}
\usepackage{bbm}

\usepackage{amsfonts,amssymb}
\usepackage{latexsym}

\newtheorem{Th}{Theorem}
\newtheorem{Prop}[Th]{Proposition}
\newtheorem{Lm}[Th]{Lemma}
\newtheorem{Co}[Th]{Corollary}

\theoremstyle{definition}
\newtheorem{Def}[Th]{Definition}

\newtheorem{Rem}{Remark}
\newtheorem{Ex}{Example}

\newcommand{\aut}{\mathrm{Aut}}
\newcommand{\reg}{\mathrm{reg}}
\newcommand{\triv}{\mathrm{triv}}
\newcommand{\autf}{\mathrm{Aut}_\mathrm{fin}}

\newcommand{\supp}{\mathrm{supp}}

\newcommand{\tr}{\mathrm{tr}}
\newcommand{\st}{\mathrm{St}^\mathrm{p}}
\newcommand{\stab}{\mathrm{St}}

\newcommand{\Sub}{\mathrm{Sub}}
\newcommand{\lin}{\mathrm{Span}}

\newcommand{\fix}{\mathrm{Fix}}

\newcommand{\rist}{\mathrm{rist}}

\newcommand{\dd}{\mathrm{d}}
\newcommand{\id}{\mathrm{Id}}

\newcommand{\eg}{\text{e.g.\,}}
\newcommand{\ie}{\text{i.e.\;}}
\newcommand{\p}{\mathrm{p}}
\newcommand{\irs}{\mathrm{IRS}}
\newcommand{\eirs}{\mathrm{EIRS}}
\newcommand{\ch}{\mathrm{Char}}
\newcommand{\ich}{\mathrm{IChar}}

\begin{document}
\title{Characters  and  IRS's  on  branch  groups and embeddings  into  hyperfinite factor}
\author{ {\bf Artem Dudko}  \\
                    IMPAN, Warsaw, Poland  \\
          adudko@impan.pl \\
         {\bf Rostislav Grigorchuk} \\        Texas A\&M University, College Station, TX, USA  \\      grigorch@math.tamu.edu }

\maketitle
\begin{abstract}
Using the  construction from  \cite{BencsToth17} of  invariant random subgroups  on weakly  branch  groups acting on regular rooted trees we  produce  uncountably  many  indecomposable  characters on  these  groups.
In fact,  we  study  three  types of  characters  coming  from the  action  of a weakly  branch  group  on  a  regular  tree, paying  attention  to their similarities  and  differences. We use obtained results to show that each countable amenable branch  group has uncountably  many pairwise not quasi-equivalent embeddings into Murray-von Neumann hyperfinite factor.
For  the  canonical character associated with a self-similar group and studied in \cite{Grig11} as a self-similar trace we  provide  a  number  of  examples  when  it  is  explicitly  computed.
\end{abstract}

\section{Introduction.}
The  goal  of this  article  is  twofold. First, to continue  the study  of  indecomposable  characters on  groups  acting  non-freely  initiated in \cite{DG15 Diag}. Second, to attract  attention  of mathematical  community  to  the invariant random subgroups, perfectly  non-free  actions,  factor representations  associated  to them,  corresponding  Murray-von Neumann-Krieger  factors,  and  embeddings into hyperfinite II$_1$ factor. We  deal  mostly  with  groups  of  branch  type  (branch  and  weakly  branch  groups) and  consider  only  the  case  of  actions  on  a  $d$-regular  tree  $T_d,  d\geq 2$.

Study  of  dynamics  of  groups acting  on  rooted trees  was  initiated  in  \cite{GNS00} and  splits  into  two  parts depending on the category in which we work:  topological  dynamical  systems  $(G, \partial  T_d)$ or  metric dynamical  systems  $(G, \partial  T_d, \mu)$.  Here $ \partial  T_d$  is a  boundary  of  the  tree  supplied  by  the natural topology  making  it  homeomorphic  to a  Cantor  set. The measure  $\mu$  is the uniform  Bernoulli  measure  on $\partial T_d$. A group  acting  by  automorphisms  on  a  tree at  the  same  time  acts  by  homeomorphisms (in fact by isometries  for a  suitable  ultrametric) on $ \partial  T_d$ preserving  the  probability measure  $\mu$. This  action  could  be  (essentially)  free or be "very  far" from a  free action. For  instance, the latter might mean that  different  points  of  the  boundary have different  stabilizers, thus   satisfying   A.Vershik's  definition  of extreme non-freeness.  Stronger  versions  of  non-freeness,  called  absolute  and  perfect   non-freeness,  were introduced  in \cite{DG15 Diag}  and  will  play an important  role  here.

A  character  $\chi$  on a  group  is a non-negatively  definite  constant  on  conjugacy  classes  and  normalized  by  $\xi(e)=1$  function, where $e\in G$ is the identity element.   Such  functions  constitute  a simplex,  extreme  points  of  which are called  indecomposable characters.  Characters  play an  important  role  for  representation  theory since the Gelfand-Naimark-Segal (GNS, for short) construction  allows  to  associate  with each  character  a unitary  representation.    At  the  same  time,  indecomposable  characters are in a  natural  bijection  with  the  classes of  quasi-equivalence of  finite type factor representations. Observe that quasi-equivalence is a  much  weaker relation  than unitary equivalence \cite{Dix69}.

One of the first and the most studied infinite groups concerning characters is the infinite symmetric group $S(\infty)$ (\ie the group of finitary bijections of the set of positive integers). Its indecomposable characters were described by Thoma in \cite{Thoma}. Later, new methods were developed, and new proofs of Thoma's result were found. In particular, Kerov and Vershik developed asymptotic theory of characters  of the symmetric group \cite{VershikKerov:1981}, \cite{VK81}, and Okunkov and Olshanski developed semigroup approach \cite{Ok99},  \cite{Olshanski:1989} to study representations of infinite symmetric group and other similar groups. This lead to a discovery of remarkable relations between characters and representations of $S(\infty)$ and various branches of math, including Combinatorics, Ergodic Theory, Probability Theory, Random Matrices, and Operator Algebras.

The  class  of  branch  groups  was  introduced  in  \cite{Gr00}  and  attracted  attention  by its connection  to just-infinite  groups  (i.e  infinite  groups  whose  every  proper  quotient is finite)  and  also  by  numerous  examples  of  groups  with  unusual  properties  like  to  be infinite finitely  generated
or  to  have  intermediate  growth  (the  group  $\mathcal G$   from  \cite{Gr00}  for  instance  is  such  group). A  larger  class  of  groups,  the  weakly branch  groups,  important  representative  of  which  is  the  Basilica  group  $\mathcal B$  introduced in  \cite{GrigZuk97},  also  attracted  a  lot  of  attention  and  is an  intermediate  class  between  branch  groups   and classes of micro-supported
groups  invented by  Rubin \cite{Rubin} (without using  this  name).  The  latter  case  includes  ample  groups  introduced  by  Krieger  \cite{Krieger:1980}, particular  case  of  which  are  topological full groups associated  with  minimal  Cantor  systems studied   by Giordano, Putnam, and Skau \cite{giordano_putnam_skau:1995}, \cite{giordano_putnam_skau:1999}, and  other  researchers.

A special feature  for  groups from the mentioned  classes  is that  for  any  open  subset  $X$ of  the  space  on  which  group  acts  there  is   a  non identity  element  $g$  of the  group  whose support  is  inside  $X$.  Thus,  with  each such  $X$  one  can  associate  a  nontrivial  subgroup $G_X$ consisting  of  elements  with  support  in  $X$,  which  in the  theory  of  branch  groups  is customary  to  call a rigid  stabilizer of  $X$.

Using  the absolute  nonfreenes  of the action  of  branch  groups  on the  boundary  of  the  tree  and  diagonal  actions the  authors  in  \cite{DG15 Diag} constructed  a  countable  family  of  indecomposable  characters.  In  this  article  we  construct  $2^{\aleph_0}$  of  such  characters.   We  use the  idea  explored  by Bencs and T\'{o}th \cite{BencsToth17} consisting  in  taking  a closure $\overline{G}$ of  a  group  $G$  acting  on  a  rooted  tree  $T_d$  in the  group  $\aut(T_d)$ of  all  automorphisms  (the  latter  is a  totally  disconnected compact  group, \ie a profinite group)  and  consider   $\overline{G}$-orbits  of closed  subsets of $\partial T_d$  together  with  the  probability  measure  on it  induced  by a  normalized  Haar  measure  on $\overline{G}$.  This  allowed  them  to build uncountably  many ergodic Invariant  Random Subgroups (IRS's)  on any   countable  weakly  branch group $G$.
We  explore  the  same idea in  combination  with non-freenes  arguments  and use  of  two  type  of  stabilizers  of  sets:  set stabilizer and pointwise  stabilizer. In fact, our  constructions  lead  to  three  types  of  characters  and  we  discuss  the  similarities  and  differences  between  them.

The obtained results  (combined  with  the  results  from  \cite{Connes:1976} and \cite{DG14})  have interesting  consequences related to the Murray-von Neumann hyperfinite factor $\mathfrak R.$
The  question   which  countable  groups  embeds  into  the  unitary  group $U(\mathfrak R)$ of  $\mathfrak R$  attracted  attention of  researchers.  It  is  known due  to the result  of A.Connes \cite{Connes:1976},  that  amenable  groups  embed.  Also  residually  amenable  groups  embed  and  any group that embeds has to be hyperlinear in the sense of R\u adulescu  \cite{Radulescu:2000} (see more on this in \cite{Pestov:2008}).
On  the  other  hand  any non amenable group with only 2 characters cannot embed. This includes  $PSL_n(\mathbb Z),n  \geq  3$   by Bekka \cite{Bekka:2007}, $SL_2(\mathbb Q)$ by Peterson-Thom \cite{PetersonThom:2016}, and any non amenable group from the commutators of Higman-Thompson groups by Dudko-Medynets \cite{DM14 Thompson}. Now  when  embedding  is  possible  a natural  question  could  be  how  many  different  embeddings  exists?

The  group  $\aut(\mathfrak R)$  of  automorphisms  of  $\mathfrak R$ preserves  $U(\mathfrak R)$  and  naturally  acts on  the  set  of  its  subgroups.  We  say  that  two  embeddings $\alpha$ and  $\beta$   of  $G$  are different  if  the  corresponding  orbits  are  different.   Restriction  of the canonical  trace  on  $\mathfrak R$ on  an embedding of $G$  gives  a  character on  $G$.  Hence,  if  $\alpha$ and  $\beta$ lead to  the  different  characters,  the  groups $\alpha(G)$ and $\beta(G)$ belong  to  different orbits.

 For instance,  taking  any  amenable  branch  group $G$ using  our  results we  get $2^{\aleph_0}$ embeddings  of  $G$  into   the  unitary  group $U( \mathfrak R)$ of  $\mathfrak R$. These embeddings are associated to characters constructed in Theorem \ref{ThMain}, part 2.
 Notice that the class of amenable weakly branch groups include \eg the group of intermediate growth $\mathcal G$ constructed by the second author,  the  Basilica $\mathcal B$ and Hanoi  Towers  Group  $\mathcal{H}^{(3)}$. Using  the  fact  the   action  of  $G$  on  the  $\bar{G}$-orbit  of a  closed subset  $C \subset  \partial T_d$ that  belongs  to  the  class  $\mathfrak C_2$ from Theorem \ref{ThMain} is  perfectly  non-free,  Proposition  \ref{PropEqIffIrred}  borrowed  from  \cite{DG14},  and  the result  of  A.Connes \cite{Connes:1976} we  get  that the  groupoid  representation $\pi_C$  associated  with  this  action  generates a  hyperfinite  factor $\mathfrak R$  and  different  sets  from   $\mathfrak C_2$  lead  to    embeddings that  belong  to  different  orbits  of  action  of  $\aut(\mathfrak R)$ on subgroups  of  $U(\mathfrak  R)$. More  on this  at  the  Section \ref{SecEmbeddings}.

The paper is organized as follows. In Section \ref{SecPreliminaries} we give necessary preliminaries on the objects involved in the paper, including groups acting on rooted trees, invariant random subgroups, characters, non-free actions, and factor representations. In Subsection \ref{SubsecMain} we present the main results of the paper. In Sections \ref{SecPsiC}-\ref{SubsecInjectivity} we give proofs of the results on characters on branch and weakly branch groups associated to invariant random subgroups. 
Section  \ref{SecCharVal} contains  computations  of  values  of  self-similar  trace  on a  number  of  branch  and  weakly  branch  groups.  In Section \ref{SecEmbeddings}  we show that  any branch group admits $2^{\aleph_0}$ different embeddings  into   the  unitary  group $U( \mathfrak R)$ of  the Murray-von Neumann hyperfinite II$_1$ factor $\mathfrak R$. 

\subsection*{Acknowledgement}
The first author acknowledges the funding by Long-term program of support of the Ukrainian research teams at the Polish Academy of Sciences carried out in collaboration with the U.S. National Academy of Sciences with the financial support of external partners. The first author was partially supported by National Science
Centre, Poland, Grant OPUS21 "Holomorphic dynamics, fractals, thermodynamic formalism",
2021/41/B/ST1/00461. The second author acknowledges the support by the Deutsche Forschungsgemeinschaft (DFG, German Research
   Foundation) -- SFB\,-TRR 358/1 2023 -- 491392403,  Humboldt  Foundation  and  University  of  Bielefeld.
   Also, both authors are grateful to the University of Geneva for their  support (Swiss NSF grant 200020-200400).
   
   The authors are very thankful to Srivatsav Kunnawalkam Elayavalli and  Florin R\u adulescu for important remarks and useful comments.

\section{Preliminaries.}\label{SecPreliminaries}
\subsection{Branch and weakly branch groups.}\label{SubsecWBgroups} Let us recall the notions of a regular rooted tree and a weakly branch group. We refer the reader to \cite{Grig11}, \cite{GNS00} and  \cite{Nekr} for the details. Throughout the paper we assume that $d\in\mathbb N$ and $d\geqslant 2$. Let $\mathcal F$ be a finite alphabet of $d$ letters. The vertices of a $d$-regular rooted tree $T_d$ can be identified with the finite words over $\mathcal F$ such that the empty word is the root of the tree. A vertex $v$ is connected to a vertex $w$ of $T_d$ if and only if $v=wa$ or $w=va$ for some letter $a\in\mathcal F$. The words of length $n\in\mathbb N\cup\{0\}$ constitute the $n$th level $V_n$ of the vertex set of $T_d$. Thus, for every $n\geqslant 1$ every vertex $v$ from $V_n$ is connected to one vertex from $V_{n-1}$ and $d$ vertices from $V_{n+1}$.

By definition, the boundary of $T_d$ is the set $\partial T_d$ of all infinite (from the right) words over $\mathcal F$. We equip $\partial T_d$ with the metric given by $\dd(x,y)=d^{-l(x,y)}$ for any $x,y\in\partial T_d$ where $l=l(x,y)$ is the maximal number such that $x_i=y_i$ for $0\leqslant i\leqslant l$. The corresponding topology coincides with the product topology and makes $\partial T_d$  a Cantor set. The group of automorphisms of $T_d$ is denoted by $\aut(T_d)$. Since the root of $T_d$ is the only vertex of degree $d$, each element $g\in\aut(T_d)$ preserves the levels $V_n$ of $T_d$ and therefore defines a continuous transformation on $\partial T_d$.

Given a vertex $v\in V_n$ denote by $T_v$ the subtree of $T_d$ consisting of all vertices of the form $vu$ (here $u$ is any finite word) and the edges joining them. Thus, for every $v$ the tree $T_v$ is naturally isomorphic to $T_d$. The space $\partial T_d$ is equipped with the unique $\aut(T_d)$-invariant Borel probability measure $\mu$. We have $\mu(\partial T_v)=d^{-n}$ for every $n\geqslant 0$ and every $v\in V_n$.

For a group $G<\aut(T_d)$ the rigid stabilizer of a vertex $v\in T_d$ is the subgroup $\rist_G(v)<G$ consisting of all elements $g\in G$ acting trivially outside of $T_v$. The rigid stabilizer of the level $n\geqslant 0$ of $T_d$ is the subgroup generated by $\rist_G(v),v\in V_n$, and is equal to an inner direct product
$$\rist_G(n)=\prod\limits_{v \in V_n}\rist_G(v),$$ since the subgroups $\rist_G(v)$ and $\rist_G(w)$ commute for any $v\neq w, v,w\in V_n$.
\begin{Def}\label{DefBranch} Let  $G<\aut(T_d)$ be an infinite group acting transitively on $V_n$ for each $n\geqslant 0$. Then $G$ is called \emph{branch} if $\rist_G(n)$ has finite index in $G$ for every $n\geqslant 0$.  $G$ is called \emph{weakly branch} if $\rist_G(v)$ is nontrivial for every vertex $v$ of $T_d$.
\end{Def}\noindent Notice that every branch group is weakly branch. For any weakly branch group $G$  $\rist_G(v)$ is conjugate to $\rist_G(w)$ for every $v,w\in V_n,n\in\mathbb N,$ and $\rist_G(v)$ is non-trivial for all $v\in V$. The latter implies that $\rist_G(v)$ is infinite for all $v\in V$.

For any $n\in\mathbb N\cup\{0\}$ and any vertex $v\in V_n$ denote by $S_v$ the finite symmetric group of all permutations of the set $V_{n+1}\cap T_v$ of vertices from $V_{n+1}$ connected by an edge to $v$. The group $S_v$ can be naturally viewed as a subgroup of $\aut(T_d)$. An important example of a countable weakly branch group is the group $\aut_{fin}(T_d)$ of all finitary automorphisms of $T_d$ generated by all subgroups $S_v,v\in T_d$.

\subsection{Invariant random subgroups.}\label{SubsecIRS}
 For a countable group $G$ denote by $\Sub(G)$ the space of all subgroups of $G$ endowed with the product topology from $\{0,1\}^G$. Every element $g\in G$ acts on $\Sub(G)$ by conjugation: $g(H)=gHg^{-1}$ for $H\in \Sub(G)$.
\begin{Def}\label{DefIRS} An \emph{invariant random subgroup} (IRS for short) of a countable group $G$ is a $G$-invariant Borel probability measure on $\Sub(G)$.
\end{Def}
\noindent
Invariant random subgroup is a natural generalization of a normal subgroup. Indeed, for any normal subgroup $H\lhd G$ of a group $G$ the delta-measure $\delta_H$ supported at $H$ is an IRS.
Invariant random subgroups arise naturally from probabity measure preserving actions. Namely, let $G$ act on a Lebesgue probability space $(X,\mu)$ by measure-preserving transformations. Consider that map
$$\stab:X\to\Sub(G),\;\;x\to\stab(x)=\{g\in G:gx=x\}.$$ The push-forward measure $\stab_*\mu$ is an IRS of $G$.

In \cite{BencsToth17}, F. Bencs and L. T\'oth constructed for every countable weakly branch group $G$ a continuum of IRS's on $G$. We briefly describe their construction. Equip $\aut(T_d)$ with the topology generated by the sets $$U_n(g)=\{h\in\aut(T_d):h(v)=g(v)\;\;\text{for any}\;\;v\in V_n\},$$ where $n\in\mathbb N,g\in\aut(T_d)$. It is not hard to see that with this topology $\aut(T_d)$ is compact and totally disconnected. Hence, $\aut(T_d)$ is a profinite group \cite{Gr00}. Given a group $G<\aut(T)$ its closure $\overline{G}<\aut(T)$ is a compact topological subgroup, therefore, $\overline{G}$ admits a unique Haar probability measure $\lambda=\lambda_{\overline G}$. Moreover, the action of $G$ on $(\overline{C},\lambda)$ by left multiplications is ergodic.

 Further, the Hausdorff distance between two subsets $C_1,C_2\partial T_d$ is given by:
  \begin{equation}\label{EqHausdorffDistance} \dd_H(C_1,C_2)=\max\left\{\sup_{x\in C_1}\dd(x,C_2),\;\sup_{y\in C_2}\dd(y,C_1)\right\}.
  \end{equation} We equip the family $\mathcal C$ of all closed subsets of $\partial T_d$ with the Hausdorff metric and the corresponding  topology.The group $\overline{G}$ acts on $\mathcal C$ by translations $B\to g(B)$ for $B\in\mathcal C,g\in\overline{G}$. This action $\overline G\times \mathcal C\to\mathcal C$ is continuous in both coordinates (see the proof of Lemma 2.3 in \cite{BencsToth17}). For $C\in\mathcal C$ denote its $\overline{G}$-orbit in $\mathcal C$ by
$[C]$. Equip $[C]$ with the push-forward $\lambda_{[C]}$ of the measure $\lambda$ via the map $\overline G\to [C],g\to gC$. The measure $\lambda_{[C]}$ is ergodic with respect to the action of $G$ as a push-forward of an ergodic measure.
\begin{Rem} If $C=\partial T_v$ for some vertex $v\in V_n,n\in\mathbb N$, then the action of $G$ on $([C],\lambda_{[C]})$ is isomorphic to the action of $G$ on the finite set $Gv\subset V_n$ equipped with the uniform probability measure. If $C=\{x\}$ for some point $x\in\partial T$, then the action of $G$ on $([C],\lambda_{[C]})$ is isomorphic to the action of $G$ on $\overline G x\subset \partial T$ equipped with the $G$-invariant probability measure. In general, the dynamical system $(G,[C],\lambda_{[C]})$ has a more complicated structure.
\end{Rem}

For a set $B\subset\partial T_d$ denote by $$\st(B)=\{g\in G:gx=x\;\;\text{for all}\;\;x\in B\}$$ its pointwise stabilizer in $G$. Consider the map \begin{equation}\label{EqPsi'}\st:\mathcal C\to\Sub(G),\;B\to\st(B)\;\;\text{for}\;\;B\in\mathcal C.\end{equation} Notice that Lemma 3.4 from \cite{BencsToth17} implies that for a weakly branch group the map $\st$ is an injection. Recall that $\mathcal C$ is equipped with the Hausdorff metric and the corresponding topology, and $\Sub(G)$ is equipped with the product topology from $\{0,1\}^G$.
\begin{Lm}\label{LmStpBorel} For any group $G$ acting on a $d$-regular rooted tree $T_d$ the map $\st:\mathcal C\to\Sub(G)$ is Borel.
\end{Lm}
\begin{proof} For $n\in\mathbb N$ introduce the map $$p_n:\mathcal C\to\mathcal C,\;\;p_n(C)=\bigcup\limits_{v\in V_n,\partial T_v\cap C\neq\varnothing}\partial T_v.$$ One has $p_n(C_1)=p_n(C_2)$ whenever $\dd_H(C_1,C_2)<d^{-n}$ (see \eqref{EqHausdorffDistance}). Therefore, the maps $p_n$ are continuous in the Hausdorff topology. In addition, the sequence $\{p_n\}_{n\in\mathbb N}$  converge pointwise to the identity map on $\mathcal C$ when $n\to\infty$. The image $p_n(\mathcal C)$ is finite, and therefore the restriction $\st|_{p_n(\mathcal C)}$ and the composition $\st\circ\p_n$ are continuous maps. Since $\st\circ p_n$ converge pointwise to $\st$, the map $\st:\mathcal C\to\mathcal \Sub(G)$ is Borel.
\end{proof}
\begin{Co} The push-forward measure \begin{equation}\label{EqMupC}\mu^{\mathrm p}_{[C]}=\st_*\lambda_{[C]}\end{equation} is an IRS of $G$.\end{Co}

 Observe also that for a closed subset $C\subset \partial T$ one can associate to the action of $G$ on $([C],\lambda_{[C]})$ another IRS denoted by $\mu_{[C]}$ using the stabilizer map
\begin{equation}\label{EqMuC}\stab:[C]\to \Sub(G),\;\;B\to \stab(B)=\{g\in G:gB=B\},\;\;\mu_{[C]}=\stab_*\lambda_{[C]}.\end{equation}
The IRS's $\mu_{[C]}$ and $\mu_{[C]}^\p$ are ergodic with respect to the action of $G$ since they are push-forwards of an ergodic measure.
In general, $\mu_{[C]}$ and $\mu^{\mathrm p}_{[C]}$ don't need to coincide. For example, if $C\subset\partial T_v$ for some $v$ then $\stab_G(C)\supset\rist_G(v)$, but $\st_G(C)\cap \rist_G(v)=\{e\}$ (here $e$ is the identity element). In this case $\mu_{[C]}\neq\mu^{\mathrm{p}}_{[C]}$. For simplicity, any closed set which is not open we call \emph{clonopen}.
 \begin{Th}[Bencs-T\'oth, \cite{BencsToth17}, Section 3.2]\label{ThBT} For any weakly branch group $G$ and any clonopen subset $C$ of $\partial T_d$ the IRS $\mu^{\mathrm p}_{[C]}$ is ergodic and continuous. Moreover, if $C_1,C_2$ are two clonopen subsets of $\partial T_d$ and  $[C_1]\neq[C_2]$ then $\mu^{\mathrm p}_{[C_1]}$ and $\mu^{\mathrm p}_{[C_2]}$ are distinct.
\end{Th}
\noindent As a result, Bencs and T\'oth obtained a continuum of ergodic IRS's for every weakly branch group.

\subsection{Characters associated to IRS.}\label{SubsecCharactersIRS}

In this paper we study relations between IRS's and characters on weakly branch groups.
\begin{Def}\label{DefChar}
A {\it character} on a group $G$ is a function $\chi:G\rightarrow
\mathbb{C}$ satisfying the following properties:
\begin{itemize}
\item[(1)] $\chi(g_1g_2)=\chi(g_2g_1)$ for any $g_1,g_2\in G$;
\item[(2)] the matrix
$\left\{\chi\left(g_ig_j^{-1}\right)\right\}_{i,j=1}^n$ is
positive semi-definite for any integer $n\geq 1$ and any elements $g_1,\ldots,g_n\in G$;
\item[(3)] $\chi(e)=1$, where $e$ is the  identity element of $G$.
\end{itemize} Extreme points in the simplex of characters are called \emph{indecomposable characters}. Equivalently, a character $\chi$ is indecomposable if and only if it
cannot be represented in the form $\chi=\alpha
\chi_1+(1-\alpha)\chi_2$, where $0<\alpha<1$ and $\chi_1,\chi_2$ are
distinct characters.
\end{Def}
\noindent The simplest examples of characters on any countable group $G$ are the \emph{trivial character} and the \emph{regular character} given by \begin{equation*}\chi_{\mathrm{triv}}(g)=1\;\;\text{for all}\;\;g\in G,\;\;\chi_{\mathrm{reg}}(g)=\delta_{e,g}=\left\{\begin{array}{ll}1,&\text{if}\;\;g=e,\\
0,&\text{otherwise}.
\end{array}\right.
\end{equation*}
The trivial character is always indecomposable. The regular character is indecomposable if and only if $G$ has the \emph{infinite conjugacy classes} property (ICC, for short), as shown in \cite{MurrayNeumann:1944}, Lemma 5.3.4. We notice that weakly branch groups are ICC  (see \cite{Grig11}, Theorem 9.17). Thus, the regular character on any weakly branch group is indecomposable. 

It is known that for any measure preserving action of a group $G$ on a probability space $(X,\mu)$ the function $\chi(g)=\mu(\fix(g))$ is a character, where $\fix(g)=\{x\in X:gx=x\}$ (see \eg \cite{V11}). In particular, for any IRS $\mu$ of a countable group $G$ the function
\begin{equation}\label{EqChiphi}\chi_\mu(g)=\mu(\{H\in\Sub(G):gHg^{-1}=H\}),g\in G\end{equation} is a character.
Notice that there is another natural character associated to an IRS $\mu$ of $G$. Introduce the following function:
\begin{equation}\label{EqChi'phi}\psi_\mu(g)=\mu(\{H\in\Sub(G):g\in H\}),g\in G.\end{equation}
The following statement is a folklore fact.
\begin{Lm}\label{LmChi'phiIsChar} For any IRS $\mu$ of a group $G$ the function $\psi_\mu$ defined by \eqref{EqChi'phi} is a character on $G$.
\end{Lm}
\begin{proof} The function $\psi_\mu$ is  central ($\psi_\mu(hgh^{-1})=\psi_\mu(g)$ for all $h,g\in G$), since $\mu$ is invariant under conjugation. One has:
\begin{equation*}\psi_\mu(g)=\int\limits_{\Sub(G)}\mathbbm{1}_H(g)\dd\mu(H)
\end{equation*} for all $g\in G$, where $\mathbbm{1}_H(g)$ is the characteristic function of $H$. Let $H<G$ be any subgroup. For any $g_1,\ldots,g_n\in G, n\in \mathbb N$ the matrix $\{\mathbbm{1}_H(g_ig_j^{-1})\}_{i,j=1}^n$ consists of zeros (whenever $g_i,g_j$ in different cosets of $G$ modulo $H$) and ones (whenever $g_i,g_j$ are in the same coset). Reorder $\{g_1,\ldots,g_n\}$ in such a way that elements of the same coset form sequences of consecutive elements. Then $\{\mathbbm{1}_H(g_ig_j^{-1})\}_{i,j=1}^n$ is a block-diagonal matrix with blocks consisting entirely of ones. Clearly, it is positive definite. Hence, for every subgroup $H<G$ the function $\mathbbm{1}_H(g)$ is positive definite. As integral of positive definite functions, $\psi_\mu$ is also positive definite. Thus, $\psi_\mu$ is a character.
\end{proof}
 In general, $\chi_\mu$ and $\psi_\mu$ do not need to coincide.
 \begin{Ex}\label{ExDistChars0} Recall that for $H\in\Sub(G)$ the symbol $\delta_H$ stands for the delta measure supported at the point $H$. By definition, $\chi_{\delta_{G}}(g)=\psi_{\delta_{G}}(g)\equiv 1$ is the trivial character on $G$. However, $\psi_{\delta_{\{e\}}}(g)=\delta_{e,g}$ is the regular character on $G$, while $\chi_{\delta_{\{e\}}}(g)\equiv 1$ is the trivial one.
  \end{Ex}
  \begin{Ex}\label{ExDistChars1} Let $T=T_2$ be a binary regular rooted tree. Let $v_1,v_2$ be the vertices of the first level. Set $C=\partial T_{v_1}$. Let $G$ be a weakly branch group acting on $T_2$ and let $a\in G$ be such that $av_1=v_2$. The IRS $\mu=\mu_{[C]}$ is concentrated at one point which is the stabilizer $\stab_G(1)=\{h\in G:h(v_1)=v_1\}$ of the first level of $\Gamma$. One has
$$\chi_\mu(a)=1,\;\;\psi_\mu(a)=0.$$
\begin{Rem}\label{DefCharsNonnegative} By definition, $\chi_\mu(g)\geqslant 0$ and $\psi_\mu(g)\geqslant 0$ for every IRS $\mu$ of a group $G$ and any $g\in G$.
\end{Rem}
Let $\irs(G)$ ($\eirs(G)$) stand for the set of invariant (ergodic invariant) random subgroups of $G$. Let $\ch(G)$ ($\ich(G)$) stand for the set of characters (indecomposable characters) on $G$.
One of the focuses of the present paper is on the properties of the two maps  $\mathcal X,\Psi:\eirs(G)\to\ch(G)$ given by
\begin{equation}\label{EqMeasuresToCharacters}
\mathcal X:\mu\to \chi_\mu,\;\;\text{and}\;\;\Psi:\mu\to\psi_\mu,\;\;\mu\in\eirs(G),
\end{equation} if $\mu\neq \delta_{G}$, where $\delta_G$ is the atomic measure supported at the point $G\in\Sub(G)$. We set $\mathcal X(\delta_G)=\Psi(\delta_G)=\chi_\reg$ (the regular character on $G$) to avoid a trivial reason for non-injectivity of $\mathcal X$ (see Example \ref{ExDistChars0}).
We focus on weakly branch groups $G$ acting on $d$-regular rooted trees $T_d$, $d\geqslant 2$.

 Some natural questions about these maps for a particular group $G$ are
\begin{itemize}\item[(i)] Do $\mathcal X$ and $\Psi$ coincide?
\item[(ii)] Is $\mathcal X$ (or $\Psi$) injective?
\item[(iii)] Is it true that $\mathcal X(\eirs(G))$ (or $\Psi(\eirs(G))$) is a subset of $\ich(G)$?
\item[(iv)] Is it true that $\mathcal X(\eirs(G))$ (or $\Psi(\eirs(G))$) contains $\ich(G)$?
\end{itemize}

For some groups $G$, for which the description of all indecomposable characters and of all invariant random subgroups is known, the answers to the above questions (or at least some of them) are also know. An example is provided by simple approximately finite groups admitting finitely many ergodic invariant measures on the boundary of the associated Bratteli diagram. The results of \cite{DM13 AF} and \cite{DM17 IRS} imply that for these groups the answers to questions (i)-(iv) are positive, \ie $\mathcal X=\Psi$ is a bijection from $\eirs(G)$ to $\ich(G)$. We will see that this is far from being true for the case of weakly branch groups in general.

For the infinite symmetric group $S(\infty)$ the indecomposable characters were first described by Thoma \cite{Thoma}. Later new proofs of the Thoma's result were obtained, new effective methods developed, and remarkable connections with other areas were found by Kerov-Vershik \cite{VK81}, \cite{VershikKerov:1981}, Okounkov \cite{Ok99}, Olshanski \cite{Olshanski:1989}, and others. Ergodic invariant random subgroups on $S(\infty)$ were described by Vershik in \cite{V12}. Both, indecomposable characters on $S(\infty)$ and EIRSs of $S(\infty)$ are parameterized by two non-increasing sequences of non-negative numbers $\alpha=\{\alpha_i\}_{i\in\mathbb N},\;\beta=\{\beta_i\}_{i\in\mathbb N}$ such that $\sum \alpha_i+\sum\beta_j\leqslant 1$. However, the maps $\mathcal X$ and $\Psi$ do not completely respect the parametrization. Namely, if $\mu_{\alpha,\beta}$ is the EIRS and $\chi_{\alpha,\beta}$ is the indecomposable character corresponding to the sequences $\alpha,\beta$, then one has:
\begin{equation*} \mathcal X(\mu_{\alpha,\beta})=\chi_{\alpha\cup\beta,0^\infty},\;\;
\Psi(\mu_{\alpha,\beta})=\tfrac{1}{2}(\chi_{\alpha,\beta}+\chi_{\alpha\cup\beta,0^\infty}).
\end{equation*}
Here $0^\infty$ is the sequence of zeros and $\alpha\cup\beta$ is the sequence obtained by merging $\alpha$ and $\beta$ non-increasingly. Thus, $\mathcal X(\eirs(S(\infty)))$ contains only indecomposable characters $\chi_{\alpha,\beta}$ with $\beta=0^\infty$ and the preimage $\mathcal X^{-1}(\chi_{\alpha,0^\infty})$ is countably infinite for every sequence $\alpha$ containing an infinite number of nonzero elements. The character $\Psi(\mu_{\alpha,\beta})$ is indecomposable if and only if $\beta=0^\infty$.

There are also examples of groups for which there are only two indecomposable characters (the trivial and the regular one) and only two EIRS ($\delta_{\{e\}}$ and $\delta_G$). These groups include, for instance, Chevalier groups over infinite discrete
fields \cite{Ovchinnikov:1971}, \cite{Ovchinnikov:1971'}, commutators of Higman-Thompson groups \cite{DM14 Thompson}, and groups of rational points in certain algebraic groups over a number field \cite{Bekka},\cite{Bekka-Francini}.  For them the properties (i)-(iv) are obviously satisfied.

We observe that, in general, non-simplicity of a group $G$ is an obstruction to a positive answer to (iv). Recall that $\mathbb R_+=[0,+\infty)$.
\begin{Lm}\label{LmNonSimpleGroups}
Leg $G$ be a countable group containing a normal subgroup $H\neq G$ such that $G/H$ is virtually abelian. Then there exists an indecomposable character $\chi$ on $G$ such that $\chi(g)\notin \mathbb R_+$ for some $g\in G$. In particular, $\chi\notin \mathcal X(\eirs(G))\cup \Psi(\eirs(G))$.
\end{Lm}
\begin{proof} Given a normal subgroup $H\lhd G$ and a character $\chi_0$ on $G/H$ the formula $\chi(g)=\chi_0(gH),g\in G$, defines a character on $G$. Thus, it is sufficient to find a character $\chi_0$ on a virtually abelien group $U=G/H$ such that $\chi_0(u)\notin \mathbb R_+$ for some $u\in U$. Moreover, without loss of generality we may assume that $U$ is either a non-trivial abelian or a non-trivial finite group.

If $U$ is abelian, then $U$ contains either a copy $Z$ of either $\mathbb Z$ or a $\mathbb Z_m=\mathbb Z/m\mathbb Z$, $m\geqslant 2$. In the first case, let $r\in \mathbb R\setminus \mathbb N$, and  $\chi_0(z+u)=\exp(2\pi irz)$, $z\in Z\simeq \mathbb Z$, $u\in U/Z$. In the second case, let $r\in\mathbb Z\setminus m\mathbb Z$ and $\chi_0(z+u)=\exp(2\pi irz/m)$, $z\in Z\simeq \mathbb Z_m$, $u\in U/Z$.

If $U$ is finite-dimensional then $$\sum\limits_{\chi\in \ich(U)}c_\chi\chi=\chi_\reg,$$ where $c_\chi=(\dim\chi)^2/|U|>0$. For any $g\in U,g\neq e$, $\chi_\reg(g)=0$ and $\chi_\triv(g)=1$, where $\chi_\triv\in \ich(U)$ is the trivial character. Thus, it is not possible that $\chi(g)\geqslant 0$ for all $\chi\in \ich(U)$, \ie there exists $\chi_0\in\ich(U)$ such that $\chi_0(g)\notin \mathbb R_+$.
\end{proof}

\subsection{Characters on weakly branch groups.}\label{SubsecWBCharacters}
Now, let us restrict our attention to the characters associated to the IRS of the form $\mu_{[C]}$ and $\mu^{\mathrm p}_{[C]}$ (see \eqref{EqMupC} and \eqref{EqMuC}). To simplify the notations, given a weakly branch group $G$ acting on a $d$-regular tree $T_d$ for a closed subsect $C\subset \partial T_d$ we denote:
\begin{equation}\label{EqCharIRSnotations} \chi_C=\chi_{\mu_{[C]}},\;\chi_C^{\mathrm p}=\chi_{\mu_{[C]}^\mathrm{p}},\;
\psi_C=\psi_{\mu_{[C]}},\;\psi_C^{\mathrm p}=\psi_{\mu_{[C]}^\mathrm{p}}
\end{equation} (see formulas \eqref{EqMupC},\;\eqref{EqMuC},\;\eqref{EqChiphi},\;\eqref{EqChi'phi}).
As Example \ref{ExDistChars1} showed, $\chi_C$ and $\psi_C$ do not need to coincide. Let us show that $\chi_C^{\mathrm p}$ and $\psi_C^{\mathrm p}$ do not need to coincide.
\end{Ex}
\begin{Ex}\label{ExDistChars2} Let $T,C,\Gamma$ be as in Example \ref{ExDistChars1} and $b\in\rist_\Gamma(v_1)\setminus\{e\}$. The IRS $\mu^\p_{[C]}$ is concentrated at two points $\rist_\Gamma(v_1)$ and $\rist_\Gamma(v_2)$. One has:
$$\chi_C^\p(b)=1,\;\;\psi_C^\p(b)=\tfrac{1}{2}.$$
\end{Ex}
To summarize, given a subgroup $G<\aut(T_d)$ for any closed subset $C$ of $\partial T_d$ we can associate two IRS's $\mu_{[C]}$ and $\mu^\p_{[C]}$ of $G$. In turn, to each of these IRS's one can associate two characters $\chi$ and $\psi$ on $G$. Thus, we obtain four possible characters on $G$ associated to each closed subset $C\subset\partial T_d$:
\begin{equation*}\chi_C,\;\chi_C^\p,
\;\psi_C,\;\psi_C^\p.
\end{equation*}
Below we provide formulas for these characters in terms of the measure $\lambda_{[C]}$ and show that $\chi^\p_C=\psi_C$ if $G$ is weakly branch.
\begin{Lm}\label{LmChiCLambda} Let $G<\aut(T_d)$ be a weakly branch group, where $d\in\mathbb N,d\geqslant 2$. Then for any clonopen subset $C\subset\partial T_d$ and any $g\in G$ one has:
\begin{align}\label{EqCharacterFormulas}\begin{split} \chi_C(g)=
\lambda_{[C]}(\{B\in [C]:\stab(B)=\stab(gB)\})\geqslant\\
\psi_C(g)=\chi^\p_C(g)=\lambda_{[C]}(\{B\in[C]:gB=B\})\geqslant \\
\psi^\p_C(g)=\lambda_{[C]}(\{B\in[C]:g|_B=\id\}.\end{split}
\end{align}
\end{Lm}
\begin{proof} Let $g\in G$. By \eqref{EqChiphi} one has:
\begin{equation*}\chi_C(g)=\mu_{[C]}(\{H<G:gHg^{-1}=H\}).\end{equation*} The IRS $\mu_{[C]}$ is concentrated on subgroups of the form $\stab(B)$, $B\in[C]$. Let $H=\stab(B)$, $B\in[C]$. Then $gHg^{-1}=\stab(gB)$.  Using \eqref{EqMuC} we obtain
the first line of \eqref{EqCharacterFormulas}.

Similarly
\begin{align*}
\psi_C(g)=\mu_{[C]}(\{H<G:g\in H\})=\lambda_{[C]}(\{B\in[C]:gB=B\}),\\
\chi^\p_C(g)=\mu^\p_{[C]}(\{H<G:gHg^{-1}=H\})=\\ \lambda_{[C]}(\{B\in[C]:\st(B)=\st(gB)\}).
\end{align*}
Now, let $B$ be a closed subset of $\partial T$ and $g\in G$ such that $B\neq gB$. Let $x\in gB\setminus B$.  There exists a vertex $v$ of $T$ such that $x\in\partial T_v$ and $\partial T_v\cap B=\varnothing$. Since $G$ is weakly branch one can find $h\in G$ such that $\supp(h)\subset\partial T_v$ and $hx\neq x$. This shows that $h\in \st(B)\setminus \st(gB)$ and so $\st(B)\neq \st(gB)$. Thus, $\st(G)=\st(gB)$ if and only if $B=gB$. This finishes the proof of the second line of \eqref{EqCharacterFormulas}.

Finally, using \eqref{EqPsi'} and \eqref{EqChiphi} we obtain:
\begin{align*}
\psi^\p_C(g)=\mu^\p_{[C]}(\{H<G:g\in H\})=\lambda_{[C]}(\{B\in[C]:g|_B=\id\}.
\end{align*} The inequalities between the lines of \eqref{EqCharacterFormulas} are straightforward.
\end{proof}
\noindent Thus, for a weakly branch group and a clonopen set $C\in\mathcal C$ we have $\psi_C=\chi_C^\p$. However, since we are interested in both maps $\mathcal X$ and $\Psi$, we will use either notation $\chi_C^\p$ or $\psi_C$ depending on whether we consider this character as image of $\Psi$ or of $\mathcal X$.

\begin{Rem}\label{RemVershikFormula} Observe that $\chi_C$ and $\psi_C$ are characters of the form $g\to\mu(\fix(g))$ for the actions of $G$ on $(\Sub(G),\mu_{[C]})$ and $([C],\lambda_{[C]})$ correspondingly, as follows from the definition of $\chi_C$ and formula \ref{EqCharacterFormulas}. In fact, $\psi_C^\p$ is also of this form for the action of $G$ on $(\overline{G}/\overline{\st(C)},\gamma_{[C]})$ by shifts, where $\gamma_{[C]}$ is the push-forward measure of $\lambda$ under the projection $\overline G\to \overline{G}/\overline{\st(C)}$ and $\overline{\st(C)}$ is the closure of $\st(C)\subset G$ in $\overline G$. Indeed, $B\in [C]$ means $B=hC$ for some $h\in\overline G$. For $g\in G$ one has \begin{align*}\gamma_{[C]}(\fix_{\overline{G}/\overline{\st(C)}}(g))=
\gamma_{[C]}(\{h\overline{\st(C)}:h\in \overline G,gh\overline{\st(C)}=h\overline{\st(C)}\}\\
=\lambda(\{h\in\overline G:h^{-1}gh\in\overline{\st(C)}\})=\lambda(\{h\in\overline G:g\in\overline{\st(hC)}\})\\
=\lambda_{[C]}(\{B\in [C]:g|_B=\id\}).
\end{align*}
Introduce the following maps:
\begin{align*}p_1:\overline{G}/\overline{\st(C)}\to [C],\;p_1(g\overline{\st(C)})=gC,\;g\in \overline G,\\
p_2:[C]\to\Sub(G),\; p_2(B)=\stab(B),\;B\in [C].
\end{align*} The three dynamical systems involved are related via $G$-equivariant pair of maps:
\begin{equation}(\overline G/\st(C),\gamma_{[C]})\overset{p_1}{\rightarrow}([C],\lambda_{[C]})\overset{p_2}{\rightarrow}(\Sub(G),\mu_{[C]}).
\end{equation}
\end{Rem}

In general, the characters $\chi_C, \chi^\p_C$ and $\psi^\p_C$ do not need to coincide.
\begin{Ex}\label{ExDistinctCharacters} Let $d=3$. Encode the vertices of $T=T_3$ by finite sequences of letters from $\{a,b,c\}$. Set
\begin{equation*} C=\bigcup\limits_{n\geqslant 0} \partial T_{c^na}=\partial T_a\cup\partial T_{ca}\cup\partial T_{c^2a}\cup \ldots .
\end{equation*} Let $g\in\aut(T_3)$ be the transposition of $T_a$ and $T_b$. For a letter $x\in \{a,b,c\}$ let $h_x\in\aut(T_3)$ be the cyclic permutation of $T_{xa},\; T_{xb}$ and $T_{xc}$ in this order. Let $G$ be any weakly branch group containing $g,h_a,h_b$, and $h_c$ (\eg $G=\autf(T_3)$). Then
\begin{align*} \chi_C(g)=1/3,\;\;\chi^\p_C(g)=\psi^\p_C(g)=0,\\
 \chi_C(h_a)=\chi^\p_C(h_a)=2/3,\;\;\psi^\p_C(h_a)=1/3.
\end{align*} Thus, the characters $\chi_C,\;\chi^\p_C$, and $\psi^\p_C$ on $G$ are pairwise distinct.
\end{Ex}
\begin{Rem}\label{RemDistinctChars} In fact, for every weakly branch group $G$ acting on $T_d$, $d\geqslant 2$, and any $C\in\mathcal C$ with nonempty interior one has $\psi_C\neq \psi_C^\p$. Indeed, there exists a vertex $v$ of $T_d$ and  $g\neq\id$ with $\supp(g)\subset\partial T_v\subset C$. Since the set of elements $\{h\in\overline{G}:hv=v\}$ has positive measure with respect to the Haar measure on $\overline{G}$, we have:
\begin{equation*}\lambda_{[C]}(\{B\in [C]:\supp(g)\subset B\})>0.\end{equation*} Using Lemma \ref{LmChiCLambda} we obtain that $\psi_C(g)>\psi_C^\p(g)$.\end{Rem}

\subsection{Non-free actions and groupoid construction.}\label{SubsecNFandGroupoid}
In \cite{V11} Vershik originated the study of non-free actions by introducing and investigating the notions of extreme non-freeness and total non-freeness. In \cite{DG15 Diag} we introduced two other useful notions of non-freeness of group actions.
\begin{Def}\label{DefANF}
 An action of a countable group $G$ on a measure space $(X,\mu)$ is called \emph{absolutely non-free} if
for every measurable set $A$ and every $\epsilon>0$ there exists $g\in G$ such that
$\mu(\fix_X(g)\Delta A)<\epsilon$.
\end{Def}
\noindent For a measure-preserving action $\alpha$ of a group $G$ on a Lebesgue probability space $(X,\Sigma,\mu)$ and a measurable set $A\subset X$ introduce the subgroups $G_{\alpha,A}<G$ of elements acting essentially trivially outside $A$:\begin{equation}\label{EqGalphaA}G_{\alpha,A}=\{g\in G:\mu(\supp(g)\setminus A)=0\}.\end{equation} We will omit the index $\alpha$ in cases when the action is clear from the context.
\begin{Def}\label{DefPNF} Let a countable group $G$ act on a Lebesgue probability space $(X,\Sigma,\mu)$ by measure-preserving transformations, where $\Sigma$ is the collection of all measurable subsets of $X$. We will say that this action is \emph{perfectly non-free} if there exists a countable collection $\mathcal A$ of measurable subsets of $X$ such that $\mathcal A$ together with the sets of zero measure generates $\Sigma$ and for each $A\in\mathcal A$ the $G_A$-orbit $\{gx:g\in G_A\}\subset X$ is infinite for $\mu$-almost all $x\in A$.
\end{Def}
\noindent The authors showed in \cite{DG15 Diag} that absolute non-freeness implies perfect non-freeness which in turn implies total non-freeness.
Notice that by Rokhlin's Theorem on Basis from \cite{Rokh49} any family $\mathcal F$ of measurable subsets of a Lebesgue probability space $(X,\Sigma,\mu)$ which separates the points of $X$ generates $\Sigma$.

Now we recall briefly the groupoid construction associated to a measure-preserving action of a group $G$ on a Lebesgue probability space $(X,\mu)$ (see \eg \cite{DG15 Diag} for details). Let $\mathcal R$ be the orbit equivalence relation on $X$ considered as a subset of $X\times X$ equipped with the action of $G$ on the left coordinate by $g(x,y)=(gx,y)$. Notice that one can consider the action on the right coordinate as well: $(x,y)\to (x,gy)$.
Let $\nu=\mu\times\{\text{counting measure on orbits}\}$ be the (infinite) $G\times G$-invariant measure on $\mathcal R$ which restricts to the diagonal $\{(x,x):x\in X\}\subset\mathcal R$ as $\mu$.
 Introduce the representations $\pi$ and $\widetilde\pi$ in $L^2(\mathcal R,\nu)$ by:
  \begin{equation}\label{EqPiTildePi}
 (\pi(g)f)((x,y))=f(g^{-1}x,y),\;\;(\widetilde\pi(g)f)((x,y))=
 f(x,g^{-1}y).\end{equation}
  The representation $\pi$ is called (left) groupoid representation of $G$.
Let $\mathcal M_\pi$ be the von Neumann algebra generated by the operators of representation $\pi$. Denote by $\mathcal M_\mathcal R$ the Murray-von Neumann (or Krieger) algebra generated by $\mathcal M_\pi$ and the operators of multiplication by the functions of the form $f(x,y)=m(x),\;m(x)\in L^\infty(X,\mu)$. Similarly, the operators of the algebra $\mathcal M_{\widetilde\pi}$ and multiplications by the functions of the form $f(x,y)=\widetilde m(y),\;\widetilde m(y)\in L^\infty(X,\mu)$ generate a von Neumann algebra ${\mathcal M}_{\widetilde{\mathcal R}}\supset \mathcal M_{\widetilde\pi}$ isomorphic to $\mathcal M_{\mathcal R}$. Set $\xi(x,y)=\delta_{x,y}$. Let
$$\mathcal H=\overline{\lin\{\pi(g)\xi:g\in G\}}.$$ Observe that by definition the triple $(\pi|_\mathcal{H},\mathcal H,\xi)$ is isomorphic to the GNS-construction associated to $$\chi(g)=(\pi(g)\xi,\xi)=\mu(\fix(g)).$$
In particular, the character $\chi(g)=\mu(\fix(g))$ is indecomposable if and only if $\pi|_\mathcal{H}$ is a factor representation. Introduce a unitary representation $\rho$ of $G\times G$ in $L^2(\mathcal R,\nu)$ by \begin{equation}\label{EqRho}\rho(g_1,g_2)=\pi(g_1)\widetilde\pi(g_2),g_1,g_2\in G.\end{equation}
Let us formulate a useful result from \cite{DG14}, Proposition 13.
\begin{Prop}\label{PropEqIffIrred} The following assertions are equivalent:\\
$1)$  $\mathcal M_{\pi}=\mathcal M_{\mathcal{R}}$;\\
$2)$ $\rho$ is irreducible;\\
$3)$ the unit vector $\xi=\delta_{x,y}$ is cyclic in $L^2(\mathcal R,\nu)$ for $\mathcal{M}_{\pi}$ (equivalently, for $\mathcal{M}_{\widetilde\pi}$).
\end{Prop}\noindent
The preprint \cite{DG14} is the first version of the paper \cite{DG15 Diag}. Proposition \ref{PropEqIffIrred} was deleted from \cite{DG15 Diag} since it was no longer used in the proofs of the main results. This proposition is useful itself and we will use it in the proof of Theorem \ref{ThEmbeddings}. For the readers convenience we provide a proof of Proposition \ref{PropEqIffIrred} in Section \ref{SecEmbeddings}.

 Let us formulate a few results from \cite{DG15 Diag} which we will use in the present paper.
\begin{Th}\label{ThWBANF} For any countable weakly branch group $G$ acting on a regular rooted tree $T_d$ the action of $G$ on $(\partial T_d,\mu_d)$ is absolutely non-free (and therefore, perfectly non-free).\end{Th}
\begin{Th}\label{ThIndecomp} Assume that the action of a countable group $G$ on a Lebesgue probability space $(X,\Sigma,\mu)$ is ergodic, measure-preserving and perfectly non-free. Let $\pi$ be the associated groupoid representation and $\mathcal M_{\mathcal R}$ be the associated Murray-von Neumann (or Krieger) algebra. Then $\mathcal M_\pi=\mathcal M_{\mathcal R}$ and the corresponding character $\chi(g)=\mu(\fix_X(g)),g\in G,$ is indecomposable.
\end{Th}
\noindent In addition, the proof of Proposition 24 from \cite{DG15 Diag} immediately implies the following:
\begin{Lm}\label{LmSeqgn} Let $G$ be a weakly branch group acting on a regular rooted tree $T_d$. Then for any clopen set $A\subset \partial T_d$ there exists a sequence of elements $g_n\in G$ with $A\subset\fix_{\partial T_d}(g_{n+1})\subset\fix_{\partial T_d}(g_n)$ for every $n$ and $$\mu_d(\bigcap\limits_{n\in\mathbb N}\fix_{\partial T_d}(g_n)\setminus A)=0.$$
\end{Lm}
\noindent Finally, let us recall a useful folklore fact (see \eg \cite{DG15 Diag}, Lemma 20):
 \begin{Lm}\label{LmFixed} Let $\kappa$ be a unitary representation of a group $\Gamma$ on a Hilbert space $H$. Set $H_1=\{\eta\in H:\kappa(g)\eta=\eta\;\;\text{for all}\;\;g\in \Gamma\}$. Then the orthogonal projection $P$ onto $H_1$ belongs to $\mathcal M_\kappa$.
 \end{Lm}
\subsection{The main results.}\label{SubsecMain}
In this section we formulate the main results of this paper. Their proofs are given in the remaining sections.

Recall that to any weakly branch group $G$ acting on a regular rooted tree $T_d$ and any closed subset $C\subset \partial T_d$ we associated the IRS's $\mu_{[C]}$ and $\mu_{[C]}^\p$ (see \eqref{EqMuC} and \eqref{EqMupC}), which in turn give rise to characters $\chi_C,\;\psi_C=\chi_C^\p$, and $\psi_C^\p$ (see Lemma \ref{LmChiCLambda}). In the present paper we focus on the first two characters. 
Our first results (Propositions \ref{Prop3=>2=>1},  \ref{PropBranch=>3b}, and \ref{PropExample}, Lemma \ref{LmEmptyInterior=>3a}, and Corollary \ref{CoEmptyInterior=>Indecomposable}) are on indecomposability of the character $\psi_C$. Observe that for a clopen set $C$ the IRS $\mu_{[C]}$ is supported on a finite set of subgroups. We are interested in continuous IRS, so let us fix a clonopen subset $C\subset\partial T$. For any clopen set $B$ set
\begin{equation}\label{EqB*}B^*=\{C_0\in[C]:B\cap C_0\;\text{is not open}\}.\end{equation} Notice that $B^*$ depends on $C$.
  For a subset $C\subset \partial T_d$  and a vertex $v\in T_d$ set $C_v=\partial T_v\cap C$. Notice that by Lemma \ref{LmChiCLambda} the character $\psi_C$ on $G$ is naturally associated to the action of $G$ on $([C],\lambda_{[C]})$. Let $V$  be the vertex set of the tree $T_d$. Consider the following conditions:
  \begin{itemize}
\item[$1)$] $\psi_C$ is indecomposable;
\item[$2)$] the action of $G$ on $([C],\lambda_{[C]})$ is perfectly non-free;
\item[$3a)$] the collection of the sets $\{(\partial T_v)^*:v\in V\}$ separates points of $[C]$;
\item[$3b)$] for any $v\in V$ and any $C_0\in(\partial T_v)^*$ the orbit $\rist_G(v)C_0$ is infinite.
\end{itemize}
We will use symbol $"\wedge"$ to denote the union of conditions. Using Theorem \ref{ThIndecomp} immediately implies:
\begin{Prop}\label{Prop3=>2=>1} For any weakly branch group acting on a regular rooted tree $T_d$, $d\geqslant 2$, and any closed subset $C\subset\partial T_d$ one has $3a)\wedge 3b)\Rightarrow 2)\Rightarrow 1).$
\end{Prop}

We do not know whether $3a)\wedge 3b)$ is equivalent to the condition $1)$ for weakly branch groups. But for branch groups we show
\begin{Prop}\label{PropBranch=>3b}
For any branch group $G$ acting on $T_d$ and any clonopen set $C\subset\partial T_d$ the condition $3b)$ is satisfied.
\end{Prop}
\begin{Rem}\label{RemBranchEmpty=>chi=phi} Proposition \ref{PropBranch=>3b} implies that for any branch group $G$ and any closed set $C$ with empty interior one has $\chi_C=\phi_C$. Indeed, by Lemma \ref{LmChiCLambda} it is sufficient to show that for any $B\in[C]$ and $g\in G$ the identity $\stab(B)=\stab(gB)$ implies $B=gB$. Assume that $gB\neq B$. Let $x\in B$ be such that $gx\notin B$. Since $B$ has empty interior, there exists a vertex $v$ such that $x\in \partial T_v$ and $\partial T_v\cap B=\varnothing$. By Proposition \ref{PropBranch=>3b} there exists $h\in\rist_G(v)$ such that $hgB\neq gB$. Then $h\in \stab(B)\setminus \stab(gB)$ and so $\stab(B)\neq \stab(gB)$. Thus, the restrictions of $\mathcal X$ and $\Psi$ on $\{\mu_{[C]}:C\in\mathcal C,\;C\;\text{has empty interior}\}$ coincide, which gives a partial positive answer to the question (i) from Subsection \ref{SubsecCharactersIRS}.
\end{Rem}

We emphasize that in the next lemma the group $G$ is not required to be branch or weakly branch.
\begin{Lm}\label{LmEmptyInterior=>3a} Let $G$ be an arbitrary subgroup of $\aut(T_d)$. Assume that a closed set $C\subset\partial T_d$ has an empty interior. Then the condition $3a)$ is satisfied.
\end{Lm}
\noindent Combining Propositions \ref{Prop3=>2=>1} and \ref{PropBranch=>3b} and Lemma \ref{LmEmptyInterior=>3a} we immediately obtain:
\begin{Co}\label{CoEmptyInterior=>Indecomposable} For any branch group $G$ acting on $T_d$ and any closed set $C\subset\partial T_d$ with empty interior the action of $G$ on $([C],\lambda_{[C]})$ by translations is perfectly-nonfree and the character $\psi_C$ is indecomposable.
\end{Co}\noindent Corollary \ref{CoEmptyInterior=>Indecomposable} shows for any branch group $G$ that \begin{equation}\Psi(\{\mu_{[C]}:C\in\mathcal C,\;C\;\text{has empty interior}\})\subset\ich(G)\end{equation} and gives a partial positive answer for the question (iii) from Subsection \ref{SubsecCharactersIRS}.

The next example shows that the condition of empty interior in Corollary \ref{CoEmptyInterior=>Indecomposable} cannot be dropped.
\begin{Ex}\label{ExPsiDecomposable} Consider the group $G=\autf(T_3)$. Let $v$ be a vertex of the first level of $T_3$ and let $x\in\partial T_3\setminus \partial T_v$. Set $C=\partial T_v\cup x$. Observe that the IRS $\mu_{[C]}$ is ergodic by Theorem \ref{ThBT}. One can check that for any $g_1,g_2,g_3\in G$
\begin{equation*}\psi_C(\sigma(g_1,g_2,g_3))=\tfrac{1}{3}\sum
\limits_{i=1}^3 \lambda(\fix(g_i)),\;\psi_C(\sigma(g_1,g_2,g_3))=0,\;\text{if}\;\sigma\neq\id.
\end{equation*} Thus, $\psi_C(g)=\chi_\reg(\sigma)\psi_1(g)$, where $\chi_\reg(\sigma)=\delta_{\sigma,\id}$ is the regular character on the group of permutations $S(V_1)$ of $V_1$, and $\psi_1(g)=\lambda(\fix(g))$. The character $\chi_\reg$ is a convex combination of three indecomposable characters on $S(V_1)$ labeled by Young diagrams:
$$\chi_\reg=\frac{1}{6}(\chi_{(3)}+4\cdot\chi_{(2,1)}+\chi_{(1,1,1)}).$$ Each of the above characters extends from $S(V_1)$ onto $G$ by the formula $\chi(\sigma(g_1,g_2,g_3)):=\chi(\sigma)$. Thus, $\psi_C$ can be represented as a convex combination of three distinct characters:
 $$\psi_C=\frac{1}{6}(\chi_{(3)}\psi_1+4\cdot\chi_{(2,1)}\psi_1+\chi_{(1,1,1)}\psi_1).$$ Thus, the character $\psi_C$ is decomposable. In addition, the latter implies that the action of $G$ on $([C],\lambda_{[C]})$ by translations is not perfectly non-free.\end{Ex} \noindent In addition, the latter example shows that $\psi_C=\psi_{\mu_{[C]}}$ does not need to be indecomposable even if the IRS $\mu_{[C]}$ is ergodic.

Next, we show that in case of a weakly branch group the situation with the character $\psi_{C}$ can be very different from the case of a branch group. Let us say that a character $\chi$ decomposes into an integral of continuum many indecomposable characters if there exists a probability space $(X,\eta)$ and a set of pairwise distinct characters $\chi_s,s\in X,$ on $G$ such that for every $g\in G$ the map $s\to \chi_s(g)$ is integrable with respect to $\eta$ and $$\chi(g)=\int\limits_{s\in X}\chi_s(g)\dd\eta(s)$$ for every $g\in G$.
\begin{Prop}\label{PropExample} There exists a weakly branch group $G$ acting on the binary rooted tree $T_2$ and a closed subset $C\subset \partial T_2$ with empty interior such that the character $\psi_C=\chi_C^\p$ decomposes into an integral of continuum many indecomposable characters. Moreover, the condition $3b)$ is not satisfied for the action of $G$ on $([C],\lambda_{[C]})$.
\end{Prop}
\noindent Thus, neither in Proposition \ref{PropBranch=>3b}, nor in Corollary \ref{CoEmptyInterior=>Indecomposable}  "branch" cannot be replaced with "weakly branch".

Our next results concern the question whether the characters associated to distinct IRS's are distinct. The following result is dedicated to the correspondence $\mu_{[C]}^\p\to \chi_C^\p$. Recall that for a weakly branch group the IRS $\mu_{[C_1]}^\p,\mu_{[C_2]}^\p$ are distinct for any two distinct classes $[C_1],[C_2]\in\mathcal C$ (see Theorem \ref{ThBT}).
\begin{Th}\label{ThMain} Let $G$ be a weakly branch group acting on a $d$-regular rooted tree $T=T_d$. Then the following holds.
\begin{itemize}
\item[$1.$] There exists a continuum $\mathfrak C_1$ of classes $[C],C\in\mathcal C$, such that for any distinct classes $[C_1], [C_2]\in\mathfrak C_1$ one has $\mu_{[C_1]}=\mu_{[C_2]}$, and so $\chi_{C_1}^\p=\psi_{C_1}=\psi_{C_2}=\chi_{C_2}^\p$. In addition, for every $[C]\in\mathfrak C_1$ the character $\chi_C^\p$ is indecomposable.
\item[$2.$] There exists a continuum $\mathfrak C_2$  of classes $[C],C\in\mathcal C$ such that for any distinct classes $[C_1], [C_2]\in\mathfrak C_2$ one has $\chi_{C_1}^\p\neq\chi_{C_2}^\p$. In addition, if $G$ is branch, then for every $[C]\in\mathfrak C_2$ the character $\chi_C^\p$ is indecomposable.
\end{itemize}
\end{Th}
\noindent In particular, we obtain that for every weakly branch group $G$ there exists a continuum family of pairwise distinct ergodic IRS's $\mu$ giving rise to the same indecomposable character $\chi_\mu$ and there exists a continuum of pairwise distinct indecomposable characters on $G$. 
Theorem \ref{ThMain} shows that the map $\mathcal X:\eirs(G)\to\ch(G)$ is "highly" non-injective. We do not know whether the map $\Psi:\eirs(G)\to\ch(G)$ is injective.

\begin{Th}\label{ThMain'} For any weakly branch group $G$ acting on a $d$-regular rooted tree $T=T_d$, $d\geqslant 2$, for any distinct classes $[C_1],[C_2],$ $C_1,C_2\in \mathcal C$, the characters $\psi^\p_{C_1}$ and $\psi^\p_{C_2}$ are distinct.
\end{Th}
\noindent
Theorem \ref{ThMain'} implies that the restriction of the map $\Psi$ to the set of IRS $\{\mu^\p_{[C]}:C\in\mathcal C\}$ of a weakly branch group $G$ is injective. This is a partial answer to question (ii) from Subsection \ref{SubsecCharactersIRS}.

\section{On characters $\psi_{C}=\chi_C^\p$}\label{SecPsiC}
\subsection{The proof of Proposition \ref{Prop3=>2=>1}}\label{SubsecProof3=>2=>1} By Theorem on Bases (see \cite{Rokh}, p. 22), the item $3a)$ implies that the collection of the sets $\mathcal A_T=\{(\partial T_v)^*:v\in T\}$ is a basis, that is in addition to separability, it generates the sigma-algebra of measurable subsets of $[C]$. Assume that $3a)$ and $3b)$ hold. Let $n\in\mathbb N$. For a vertex $v\in V_n$ such that $\partial T_v\cap C$ is open let $L_v$ be the minimal number such that $\partial T_v\cap C$ is a union of cylinders of level $L_v$. Let $L_v=\infty$ when $\partial T_v\cap C$ is clonopen. Set
$$N=\sup\{L_v:v\in V_n, L_v\neq\infty\}.$$

Now, fix $v\in V_n$. For any $C_0\in [C]$ the subgroup  $\rist_G(v)\cap \stab_G(N)$ fixes $C_0$ whenever $C_0\cap\partial T_v$ is open. It follows that $\rist_G(v)\cap \stab_G(N)<G_{(\partial T_v)^*}$ (see \eqref{EqGalphaA}).
Observe that $\rist_G(v)\cap \stab_G(N)$ is a finite index subgroup in $\rist_G(v)$. If $C_0\cap\partial T_v$ is not open, using $3b)$ we conclude that the orbit of $C_0$ under $\rist_G(v)\cap \stab_G(N)$ is infinite.
Then, from the definition of perfect-nonfreeness we obtain that the action of $G$ on $[C]$  is perfectly non-free (with respect to the collection of subsets $\mathcal A_T$ of $[C]$). Since the action of a weakly branch group on $[C]$ is ergodic, by Theorem \ref{ThIndecomp}, the condition $2)$ implies the condition $1)$.

\subsection{The proof of Proposition \ref{PropBranch=>3b}}\label{SubsecProofBranch=>3b}
Let $G$ be a branch group acting on a $d$-regular rooted tree $T_d$. Let $v\in T_d$ and let $C\subset \partial T_d$ be a closed set such that $C\cap\partial T_v$ is clonopen. 
Then there exists a point $x\in \overline C\cap \overline{\partial T_v\setminus C}$. Let $n$ be such that $v\in V_n$. Since $G$ is branch, $Gx$ is dense in $\partial T_d$ and $G/\rist_G(n)$ is finite. It follows that $\rist_G(n)x$ is dense in some open set $A\subset \partial T_v$. Let $w\in T_v$ be a vertex such that $\partial T_w\subset A$ and $x\notin\partial T_w$. There exists $g\in G$ such that $gx\in\partial T_w$. Since $g^{-1}\rist_G(n)g=\rist_G(n)$ we obtain that $\rist_G(v)x$ is dense in $\partial T_{g^{-1}w}$.

Further, since $x\in \overline{\partial T_v\setminus C}$ and $C$ is closed, there exists a vertex $u_1\in T_{g^{-1}w}$ such that $\partial T_{u_1}\cap C=\varnothing$. Since $\partial T_{u_1}\subset A$, there exists $h_1\in\rist_G(v)$ such that $h_1x\in \partial T_{u_1}$. Clearly, $h_1C\neq C$.

Now, by induction one can construct a sequence of vertexes $u_n\in T_d$ and a sequence of elements $h_n\in \rist_G(v)$ such that the following holds for every $n\in\mathbb N$.
\begin{itemize}
\item{} $u_{n+1}\in T_{u_n}$,
\item{} $\partial T_{u_{n+1}}\cap h_nC=\varnothing$,
\item{} $h_{n+1}x\in \partial T_{u_{n+1}}$.
\end{itemize}
Indeed, the induction assumption implies that $h_nx\in \overline{\partial T_{u_n}\setminus h_nC}$. Since $C$ is closed, there exists $u_{n+1}\in T_{u_n}$ such that $\partial T_{u_{n+1}}\cap h_nC=\varnothing$. Since $\rist_G(v)x$ is dense in $\partial T_{g^{-1}w}$ there exists $h_{n+1}\in\rist_G(v)$ such that $h_{n+1}x\in\partial T_{u_{n+1}}$.

The above properties imply that for every $n>m$ one has $h_nx\notin h_mC$ and therefore $h_nC\neq h_mC$. This shows that the orbit $\rist_G(v)C$ is infinite and proves the condition $3b)$.
\subsection{The proof of Lemma \ref{LmEmptyInterior=>3a}}
Let $G<\aut(T_d)$ and let $C\subset \partial T_d$ be a closed set with empty interior. Let $C_1,C_2\in [C]$, $C_1\neq C_2$. Then $C_1\setminus C_2$ is not-empty. Let $x\in C_1\setminus C_2$. Since $C_2$ has empty interior there exists a vertex $v$ of $T_d$ such that $x\in\partial T_v\subset \partial T_d\setminus C_2$. By definition, $C_1\in (\partial T_v)^*$ but $C_2\notin(\partial T_v)^*$. This proves that $(\partial T_v)^*$ separates the points of $[C]$.

\subsection{The proof of Proposition \ref{PropExample}}
\label{SubsecExamplePsiC} Let $d=2$. For a vertex $v$ of $T$ let $\sigma_v$ be the transposition of the two edges descending from $v$.  Consider the group $G$ generated by all elements of the form:
 \begin{equation}\label{EqGen}1)\;\;\sigma_v,\;v\in V_l,\;\;l\;\;\text{is even};\;\;\;2)\;\;
 h_l=\prod\limits_{v\in V_l}\sigma_v,\;\;l\;\;\text{is odd}.\end{equation}  As we show in \cite{DG15 Diag}, proof of Proposition 26, $G$ is weakly branch but not branch.

 Identify the alphabet $\mathcal F$ with $\{0,1\}$. Recall that $\partial T$ is identified with infinite words over $\mathcal F$, that is with sequences $\{a_j\}_{j\in\mathbb N}$, $a_j\in\{0,1\}$. Let
$$C=\{\{a_j\}_{j\in\mathbb N}\in\partial T:a_{2i}=0\;\;\text{for all}\;\;i\in\mathbb N\}.$$ Then $[C]$ consists of all sets of the form $B=B(t)=\{\{a_j\}:a_{2i}=t_i\;\;\text{for all}\;\;i\in\mathbb N\}$, where $t\in\{0,1\}^\infty$. Observe that by Theorem \ref{ThBT}, the IRS $\mu_C^\p$ is ergodic and continuous

Consider now an element $g\in G$. By construction, $g$ can be written as $\sigma\times h$, where $\sigma=\sigma_{v_1}\cdots \sigma_{v_k}$ is a product of distinct elements of type $1)$ and $h=h_{l_1}\cdots h_{l_m}$ is a product of distinct elements of type $2)$ from \eqref{EqGen}. It is straightforward to verify that $\sigma$ acts trivially on $[C]$ and $h$ either acts freely on $[C]$ or is equal to the identity element $e$ of $G$.

Further, for $g=\sigma\times h$ as above,
$$\psi_C(g)=\left\{\begin{array}{ll}1,&\text{if}\;\;h=e,\\
0,&\text{otherwise}.
\end{array}\right.$$ For each $s\in\{-1,1\}^{\mathbb Z_+}$ introduce a homomorphism $\chi_s$ from $G$ to the two-element multiplicative group $\{-1,1\}$ by defining $\chi_s$ on the generating elements as follows:
$$\chi_s(\sigma_v)=s_i\;\;\text{for all}\;\;v\in V_{2i},\;\;\chi_s(h_{2i+1})=1,\;\;\text{for all}\;\;i\geqslant 0.$$ Notice that $\chi_s$ is an indecomposable character of $G$ for each $s\in \{-1,1\}^\infty$. Let $\lambda=\{\tfrac{1}{2},\tfrac{1}{2}\}^\infty$ be the uniform Bernoulli measure on $\{-1,1\}^\infty$. Then one has:
$$\psi_C(g)=\int\limits_{s\in\{-1,1\}^\infty}
\chi_s(g)\dd\lambda(s).$$

In addition, the condition $3b)$ is not satisfied for this group $G$. Indeed, let $v\in V_1$. Then $\rist_G(v)$ is generated by the elements $\sigma_w$, $w\in T_v\cap V_l$, where $l$ is even. It follows that $\rist_G(v)$ acts trivially on $C$, however $C\cap\partial T_v$ is not open, and thus $C\in(\partial T_v)^*$.

Observe that $\stab(C)=<\{\sigma_v:v\in V_l,l\;\text{is even}\}>$. It follows that $\stab(C)$ is a normal subgroup of $G$ and thus $\chi_C(g)=1$ for all $g\in G$.

\section{On the correspondence $\mu_{[C]}^\p\to\chi_C^\p$: the proof of Theorem \ref{ThMain}}
\subsection{Part $1)$.}\label{Subsec4}
 Let $G$ be a weakly branch group acting on a regular rooted tree $T=T_d$. Denote by $\mathcal N$ the set of all increasing sequence $\bar n=\{n_j\}_{j\in\mathbb N}$ of positive integers such that $n_{j+1}>n_j+1$ for infinitely many $j$. Fix $\bar n\in\mathcal N$. For a point $x\in\partial T$ and $j\in\mathbb N$ denote by $v_j=v_j(x)$ the vertex from $V_{n_j}$ such that $x\in\partial T_{v_j}$. Similarly, let $w_j=w_j(x)$ be the vertex from $V_{n_j+1}$ such that $x\in\partial T_{w_j}$. Set
$$C(x)=\bigcup\limits_{j\in\mathbb N}(\partial T_{v_j}\setminus \partial T_{w_j}) \cup \{x\}.$$ Existence of infinitely many $j$ with $n_{j+1}>n_j+1$ guarantees that $C(x)\neq C(y)$ for $x\neq y$. Notice that for any $g\in\aut(T)$ one has $g(C(x))=C(gx)$. Therefore, $[C(x)]=\{C(y):y\in\partial T\}$ for any $x$. Denote $C_{\bar n}=[C(x)]$. Then the action of $G$ on $C_{\bar n}$ is isomorphic to the action of $G$ on $\partial T$. Since $\mu=\mu_d$ is the unique $G$-invariant measure on $\partial T$, we obtain that $\lambda_{C_{\bar n}}$ is the push-forward of the measure $\mu$. It follows that the dynamical systems $(G,C_{\bar n},\lambda_{C_{\bar n}})$ and $(G,\partial T,\mu)$ are isomorphic and so by Lemma \ref{LmChiCLambda} $\psi_{C(x)}
=\chi_1$, where $\chi_1(g)=\mu(\fix_{\partial T}(g))$ for every $g\in G$. From \cite{DG15 Diag}, Corollary 17, we deduce that $\chi_1$ is indecomposable. Notice that for any $\bar n\in\mathcal N$ and any $x\in\partial T$ the sequence of vertices $w_j(x)$ (and therefore the sequence $\bar n$) can be recovered from the set $C(x)$. It follows that the classes $C_{\bar n}\subset\mathcal C,\bar n\in\mathcal N,$ are pairwise disjoint. This finishes the proof of part $1)$ of Theorem \ref{ThMain}.

\subsection{Part $2)$.}
As before, let $G$ be a weakly branch group acting on a regular rooted tree $T$. As in the proof of part $4)$ we start with a sequence $\bar n\in\mathcal N$ and a point $x\in\partial T$. Let $v_j$ and $w_j$ be as in Subsection \ref{Subsec4}. For each $j\in\mathbb N$ fix a point $x_j=x_j(\bar n,x)\in \partial T_{v_j}\setminus \partial T_{w_j}$ and set $$C'=C'(\{x_j\})=\{x_j:j\in\mathbb N\}\cup\{x\}.$$ Since the sequence $\{x_j\}$ accumulates to $x$ the set $C'$ is a countable closed set. Let $C'_{\bar n}=[C']$. Notice that $C'_{\bar n}$ may depend on the choice of $x$ and $x_j$. For our purposes it is sufficient to fix one choice of $x$ and $x_j$ for every $\bar n\in\mathcal N$ (\eg by picking the smallest possible $x$ and $x_j$ in the lexicographic order). Observe that for every $g\in\overline G$ the point $gx$ is the unique limit point of $gC'$ and for every $j$ one has $\dd(gx,gx_j)=d^{-n_j}$. It follows that for distinct $\bar n$ the classes $[C'(\{x_j\})]$ are distinct.

Moreover, for every $C_1=gC'\in[C']$, $g\in \overline G$, the stabilizer $\{h\in G:hC_1=C_1\}$ coincides with the pointwise stabilizer $\{h\in G:hx=x\;\;\text{ for every}\;\;x\in C_1\}$. Indeed, if $hC_1=C_1$ then $h$ preserves the unique limit point of $C_1$. Since the distances from other points of $C_1$ to the limit point are pairwise distinct and $h$ preserves these distances, it follows that $h$ fixes all other points of $C_1$ as well.
 From Lemma \ref{LmChiCLambda} we obtain that $\chi_{C'}=\psi_{C'}=\psi^\p_{C'}$. In particular, by Theorem \ref{ThMain'}, for distinct $\bar n\in \mathcal N$ the characters $\psi_{C'(\{x_j\})}$ are distinct.

 Indecomposability of the characters $\psi_{C'(\{x_j\})}$ in the case of a branch group follows from Corollary \ref{CoEmptyInterior=>Indecomposable}.

\section{Injectivity of the map $[C]\to\psi_{[C]}^\p$: the proof of Theorem \ref{ThMain'}.}\label{SubsecInjectivity}
Let $C\in \mathcal C$. For $B\in\mathcal C$ set $B_*=\{C\in [C]:C\subset B\}$.
\begin{Lm}\label{LmTildeA} Let $G$ be a weakly branch group acting on a regular rooted tree. Let $C$ be a closed subset of $\partial T$. Let $A\subset \partial T$ be clopen and $\widetilde A\supset A$ be closed such that $\mu(\widetilde A\setminus A)=0$. Then $$\lambda_{[C]}(\widetilde A_*)=\lambda_{[C]}(A_*).$$
\end{Lm}
\begin{proof} Let $G_A=\{g\in G:gA=A\}$. Since $A$ is clopen, the $G$-orbit of $A$ in $\mathcal C$ is finite, therefore $G_A$ has finite index in $G$. Denote this index by $N=[G:G_A]$.

Let $n_0\in\mathbb N$ be such that $A$ is a union of some cylinders $\partial T_v, v\in V_{n_0}$. For $n\geqslant n_0$ denote by $U_n$ and $W_n$ the sets of vertices $v\in V_n$ such that $\partial T_v\cap(\widetilde A\setminus A)\neq\varnothing$ and $\partial T_v\cap A=\varnothing$ correspondingly. Then $U_n\subset W_n$ and $$\lim\limits_{n\to\infty}\frac{|U_n|}{|W_n|}=0.$$ For a set $C_1\in [C]$ denote by $Y_n(C_1)$ the set of vertices $v\in V_n$ such that $C_1\cap \partial T_v\neq\varnothing$. Then one has:
$$ \widetilde A_*\setminus A_*\subset \{C_1\in[C]:Y_n(C_1)\subset U_n,Y_n(C_1)\neq\varnothing\}.$$
Since $G$ is level-transitive, for every $v\in W_n$ one has $|G_Av|\geqslant |W_n|/N$, where $G_Av$ is the $G_A$-orbit of $v$ in $V_n$. Therefore, for every non-empty subset $Y\subset U_n$ one has:
$$|G_AY\cap\{Z\subset U_n\}|\leqslant N|U_n||G_AY|/|W_n|,$$ where $G_AY$ is the $G_A$-orbit of $Y$ in the space of subsets of $V_n$. Since $\lambda_{[C]}$ is $G$-invariant we obtain:
\begin{align*}\lambda_{[C]}(\{C_1\in [C]:Y_n(C_1)\subset U_n, Y_n(C_1)\in G_AY\})\leqslant \\
\frac{N|U_n|}{|W_n|}\lambda_{[C]}(\{C_1\in [C]:Y_n(C_1)\in G_AY\}).\end{align*} Summing up the above inequality with $Y$ running over representatives of $G_A$-orbits of non-empty subsets of $U_n$ we arrive at
$$\lambda_{[C]}(\{C_1\in [C]:Y_n(C_1)\subset U_n,\})\leqslant \frac{N|U_n|}{|W_n|}.$$ Letting $n$ go to infinity we obtain:
$$\lambda_{[C]}(\widetilde A_*\setminus A_*)=0,$$ which finishes the proof.
\end{proof}
  Let $g_n$ be the sequence of elements from Lemma \ref{LmSeqgn}. Set  $$A_n=\fix_{\partial T_d}(g_n),\;\;\widetilde A=\bigcap\limits_{n\in\mathbb N}A_n.$$ We have that $A\subset A_{n+1}\subset A_n$ for every $n$ and $\mu_d(\widetilde A\setminus A)=0.$ From definition of $(A_n)_*$ we have $$\bigcap\limits_{n\in\mathbb N}(A_{n})_*=\widetilde A_*.$$ Fix a closed set $C\subset\partial T$.  We have:
 \begin{equation}\label{EqLimChiC}\begin{split}0\leqslant \chi'_{\tilde\phi_{[C(g_n)]}}-\lambda_{C}(A_*)=\lambda_{C}((A_n)_*)-\lambda_{C}(A_*)\\ \to \lambda_{C}(\widetilde A_*)-\lambda_{C}(A_*)=0,\end{split}\end{equation}
by Lemma \ref{LmTildeA}.

Assume now there exists $[C_1]\neq [C_2]$ such that $\psi^\p_{C_1}=\psi^\p_{C_2}$. Using \eqref{EqLimChiC} we obtain that $\lambda_{[C_1]}(A_*)=\lambda_{[C_2]}(A_*)$ for every clopen subset $A\subset \partial T_d$. Consider $\lambda_{[C_1]}$ and $\lambda_{[C_2]}$ as measures on $[C_1]\cup [C_2]$. The family of the sets $\mathcal F'=\{B_*:B\subset \partial T\;\;\text{is a clopen set}\}$ separates the points of $[C_1]\cup [C_2]$ and for any $(B_1)_*,(B_2)_*\in \mathcal F'$ one has $(B_1)_*\cap (B_2)_*=(B_1\cap B_2)_*\in \mathcal F'$. Therefore, by \cite[Corollary 1.6.3]{Cohn},  we conclude that $\lambda_{[C_1]}=\lambda_{[C_2]}$ which is impossible since the two measures are supported on disjoint sets. This contradictions shows that the characters $\psi_C^\p$ are pairwise distinct.

\noindent
\section{Values of the canonical character on elements of some contracting weakly branch groups.}\label{SecCharVal}
Given a countable self-similar group $G$ acting on the $d$-regular tree $T_d$ let $\chi_1(g)=\mu(\fix(g))$, where $\mu$ is the unique $\aut(T_d)$-invariant measure on $\partial T_d$, namely the uniform Bernoulli measure given by the distribution $\{\tfrac{1}{d},\ldots,\tfrac{1}{d}\}$ on the alphabet $\mathcal F$. The lower index $1$ in $\mathcal \chi_1$ is to indicate that the measure of fixed points is raised to power $1$, since for any $k$ the formula $\chi_k(g)=\mu(\fix(g))^k$, $g\in G$, also defines a character on $G$. We will use the notation $\chi_1(G)=\{\chi_1(g):g\in G\}$. In this section we describe the sets $\chi_1(G)$ for several important branch and weakly-branch groups. 

The above goal is motivated by the following observations. Let $(\pi_1,\mathcal H_1,\xi_1)$ be the GNS-construction associated to the $\chi_1$ on a weakly branch group $G$ acting on $T_d$, $d\geqslant 2$. Let $\tr_1$ be the canonical trace on the von Neumann algebra $\mathcal M_{\pi_1}$. As before, to a subset $A\subset \partial T_d$ we associate the subgroup $G_A=\{g\in G:\supp(g)\subset A\}$. Denote by $P_A$ the orthogonal projection onto the subspace $\{\eta\in\mathcal H_1:\pi(g)\eta=\eta\;\text{for each}\;g\in G_A\}$ of $\mathcal H_1$. It is well known that for every $g\in G$ one has (see \cite{DG15 Diag}):
\begin{equation*}\chi_1(g)=(\pi_1(g)\xi_1,\xi_1)=\tr_1(P_{\supp(g)}).
\end{equation*} Thus, $\chi_1(G)=\{\tr_1(P_{\supp(g)}):g\in G\}$ is equal to the set of values of the trace $\tr_1$  of $\mathcal M_{\pi_1}$ on a natural collection of orthogonal projections in $\mathcal M_{\pi_1}$. The latter is related to the dimension theory for $\mathcal M_{\pi_1}$ (see \cite{MurrayNeumann:1936}).
\begin{Rem} Recall that by Theorem \ref{ThWBANF} weakly branch groups act on the boundary of $\partial T_d$ absolutely non-freely. By definition, this means that for any measurable subset $A\subset\partial T_d$ and any $\epsilon>0$ there exists $g\in G$ such that $\mu(\fix(g)\triangle A)<\epsilon$. Hence, the set $\chi_1(G)$ is dense in $[0,1]$.  An interesting question is what can be said about the shape of the set $\overline{\chi_1(G)}$ for other classes of groups. For example, are there natural groups $G$ for which this set is $a)$ a Cantor set, $b)$ a union of $n>2$ intervals, $c)$ a set of the from $\{x_i\}_{i\in\mathbb N}\cup\{x\}$, where $\{x_i\}_{i\in\mathbb N}$ is a sequence of distinct points converging to $x$? Further  studies  in this  direction  could  share more  light  on  the  structure  and  properties  of  group  characters. 
\end{Rem}

Recall that a self-similar group $G$ is called \emph{contracting} if there exists a finite set $\mathcal N\subset G$ such that for every $g\in G$ there exists $k\in\mathbb N$ such that $g|_v\in\mathcal N$ for every $v\in V_n$ for every $n\geqslant k$ (see \eg \cite{Nekr}). The minimal set $\mathcal N$ with this property is called the \emph{nucleus} of $G$. 
Let us call the \emph{core} of $G$ the following set:
$$\mathrm{Core}(G)=\{g\in G:l(g_v)\geqslant l(v)\;\text{for some}\;v\in V_1\},$$ where $l(v)$ is the length of $g$ with respect to the system of generators associated to the self-similar structure. We do not claim that $\mathrm{Core}(G)$ coincides with the nucleus of $G$. Observe that $e\in \mathrm{Core}(G)$. Set $\mathbb Q_d=\{p/d^n:p\in\mathbb Z,n\in \mathbb N\}$.
\begin{Lm}\label{LmChi1Values} Assume that a self-similar group $G$ acting on the $d$-regular tree $T_d$ is contracting. Then \begin{equation}\label{EqValuesChi1G}\chi_1(G)\subset \{\sum_{s\in \mathrm{Core}(G)}p_s\chi_1(s):p_s\in \mathbb Q_d,p_s\geqslant 0,\sum_{s\in S\cup \{e_G\}}p_s=1\}.\end{equation}
\end{Lm}
\begin{proof} We prove Lemma \ref{LmChi1Values} by induction on the length of an element $g\in G$. The unique element $e_G$ of length zero satisfies the statement of the lemma, since $\chi_1(e_G)=1$. Assume that the statement of the lemma is true for all elements of $G$ of length at most $k,k\in\mathbb N\cup\{0\}$. Let $g$ be an element of length $l(g)=k+1$. If $g\in \mathrm{Core}(G)$, then $\chi_1(g)$ obviously belongs to the right hand side of \eqref{EqValuesChi1G}. Otherwise, consider the representation of $g$ in terms of sections over the first level $V_1$ of $T_d$:
$$g=(g_{v_1},\ldots,g_{v_d})\sigma,\;\;\sigma\in S(V_1),\;g_{v_1},\ldots,g_{v_d}\in G.$$
We have $$\chi_1(g)=\tfrac{1}{d}\sum_{v\in V_1:\sigma(v)=v} \chi(g_v).$$ Since $l(g_v)<l(g)$ for each $v\in V_1$ by induction we obtain the desired statement.
\end{proof}
The following statement is straightforward.
\begin{Lm}\label{LmChi1Values2} Let $G$ be a group acting on a $d$-regular rooted tree $T_d$. Assume that $G$ is weakly regular over a subgroup $K$. Then
$$\chi_1(G)\supset \chi_1(K)\supset \tfrac{1}{d}\{\sum_{i=1}^{d}r_i:r_i\in\chi_1(K)\}.$$
\end{Lm}
\noindent Applying the above two lemmas we will compute the sets of values $\chi_1(G)$ for several examples of groups acting on rooted trees.

The first example of such computations was given by R. Kravchenko for the group $\mathcal G=<a,b,c,d>$ of intermediate growth (between polynomial and exponential) constructed in \cite{Gr80}. It is presented in \cite{Grig11} where it is shown that $\chi_1(\mathcal G)=\tfrac{1}{7}\mathbb{Q}_2\cap [0,1]$.
\paragraph{Basilica group.} The Basilica group $\mathcal B$ was introduced (without a name) in \cite{GrigZuk97},\cite{GZ02}. It is generated by two automorphisms $a,b$ of $T_2$ satisfying the recursions $a=(1,b),\;b=\sigma(1,a)$. It has many interesting properties. In particular, $\mathcal B$ is torsion free, of exponential growth, amenable but not  subexponentially amenable \cite{BartholdiVirag-amenability}, \cite{GZ02}, has trivial Poisson boundary, is
weakly branch but not branch. Observe that Basilica is a self-similar contracting group, $\mathrm{Core}(\mathcal B)=\{e,a,b\}$, $\chi_1(a)=1/2$, $\chi_1(b)=0$, from Lemma \ref{LmChi1Values} we obtain $\chi_1(\mathcal B)\subset \mathbb Q_2\cap [0,1]$. Since $\mathcal B$ is regular weakly branch over the commutator subgroup $\mathcal B'$ (see \cite{GZ02}) and there exists $g=aba^{-1}b^{-1}\in \mathcal B'$ with $\chi_1(g)=0$, from Lemma $\ref{LmChi1Values2}$ we obtain $\chi_1(\mathcal B)\supset \mathbb Q_2\cap [0,1]$. Thus, $\chi_1(\mathcal B)=\mathbb Q_2\cap [0,1]$.
\paragraph{Hanoi Towers group $H^{(3)}$.} Recall that Hanoi towers group on three pegs, introduced in \cite{GrigorchukSunik-asymptotic}, is the group $H^{(3)}=<a,b,c>$ acting on $T_3$ with the generators $a,b,c$ satisfying the recursions $a=(1,1,a)(1 2),\;b=(1,b,1)(1 3), c=(c,1,1)(2 3)$. It is self-similar, contracting, and branching over its commutator subgroup $K$. It is straightforward to verify that $\mathrm{Core}(H^{(3)})=\{e,a,b,c\}$, $\chi_1(a)=\chi_1(b)=\chi_1(c)=0$. By Lemma \ref{LmChi1Values} we obtain that $\chi_1(H^{(3)})\subset \mathbb Q_3$. On the other hand, straightforward computations show that $K\ni (ab)^2=(b,a,ab)(1 2 3)$ and $\chi_1((ab)^2)=0$. Thus, $\chi_1(K)\supset\{0,1\}$. Using Lemma \ref{LmChi1Values2} we obtain $\chi_1(K)=\mathbb Q_3\cap [0,1]$.
\paragraph{The group IMG$(z^2+i)$} is an Iterated Monodromy Group (IMG) of the polynomial $z^2+i$ which was considered for the first time in \cite{BartholdiGrigorchukNekrashevych-fractal-03}. It is a regular branch group of intermediate growth generated by 4-state automaton and studied in \cite{BartholdiGrigorchukNekrashevych-fractal-03} and \cite{BuxPerez-IMG-06}. This group is generated by the automorphisms $a,b,c$ of the binary rooted tree $T_2$ such that $$a=(1,1)\sigma,\;b=(a,c),\;c=(b,1).$$ It is not hard to check that $\mathrm{Core}(\mathrm{IMG}(z^2+i))=\{e,a,b,c,bc,cb\}$. Taking into account the recursive relations we obtain:
$$\chi_1(a)=0,\;\chi_1(b)=\frac{\chi_1(c)}{2},\;\chi_1(c)=\frac{\chi_1(b)+1}{2}.$$
It follows that $\chi_1(b)=1/3,\;\chi_1(c)=2/3$. Further, $bc=(ac,c)$, therefore $\chi_1(bc)=\chi_1(c)/2=1/3$. Similarly, $\chi_1(cb)=1/3$. Using Lemma \ref{LmChi1Values} we obtain that $\chi_1(\mathrm{IMG}(z^2+i))\subset\tfrac{1}{3}\mathbb Q_2\cap [0,1]$.

The group IMG$(z^2+i)$ is regular over the subgroup $K=<[a,b],[b,c]>^{\mathrm{IMG}(z^2+i)}$. One has
$$[a,b]=aba\cdot b=(ac,ca)\;\Rightarrow\;\chi_1([a,b]=0),\;\;[a,c]=aca\cdot a=(b,b)\;\Rightarrow\;\chi_1([a,c])=1/3.$$ Using Lemma \ref{LmChi1Values2} we obtain that $\chi_1(\mathrm{IMG}(z^2+i))\supset \tfrac{1}{3}\mathbb Q_2\cap [0,1]$. Therefore, $$\chi_1(\mathrm{IMG}(z^2+i))=\tfrac{1}{3}\mathbb Q_2\cap [0,1].$$
\paragraph{Overgroup $\widetilde{\mathcal G}$} was studied in \cite{BG} and \cite{BG02}. It contains $\mathcal G$, has intermediate growth, but is much larger than $\mathcal G$. The group $\widetilde{\mathcal G}$ shares with $\mathcal G$ many other properties and naturally appear as a subgroup of a topological full group of a minimal subshift.

The group $\widetilde{\mathcal G}$ is generated by four automorphisms $a,b,c,d$ of $T_2$ given by the recursions $$a=\sigma,\;b=(a,c),\;c=(1,d),\;d=(1,b).$$ Using the recursions we obtain:
$$\chi_1(a)=0,\;\chi_1(b)=\tfrac{1}{2}\chi_1(c),\;\chi_1(c)=\tfrac{1}{2}(\chi_1(d)+1),\;
\chi_1(d)=\tfrac{1}{2}(\chi_1(b)+1).$$
It follows that $\chi_1(b)=3/7,\;\chi_1(c)=6/7,\;\chi_1(d)=5/7$.  Further, using that $b,c,d$ are pairwise commuting and $b^2=c^2=d^2=1$ one can find that $\widetilde{\mathcal G}$ is contracting with $$\mathrm{Core}(\widetilde{\mathcal G})=\{e,a,b,c,d,bc,bd,cd,bcd\}.$$ Computations similar to the above show that $$\chi_1(bc)=\tfrac{2}{7},\;\chi_1(bd)=\tfrac{1}{7},\;\chi_1(cd)=\tfrac{4}{7},\;\chi_1(bcd)=0.$$
Using Lemma \ref{LmChi1Values} we arrive at $\chi_1(\widetilde{\mathcal G})\subset\tfrac{1}{7}\mathbb Q_2\cap [0,1]$.

On the other hand, $\widetilde{\mathcal G}$ is regular branch over the subgroup $$\widetilde K=<(ab)^2,(ad)^2>^{\widetilde{\mathcal G}}.$$ Notice that \begin{align*}(ab)^2=(ca,ac),\;\chi_1((ab)^2)=0,\;(ad)^2=(b,b),\;
\chi_1((ad)^2)=\tfrac{3}{7},\\ (ad)^2(ac)^2=(bd,bd),\;
\chi((ad)^2(ac)^2)=\tfrac{1}{7}.\end{align*} It follows that $\chi_1(\widetilde K)\supset\{0,\tfrac{1}{7},\tfrac{3}{7},1\}$. Applying Lemma \ref{LmChi1Values2} a few times we conclude that $\chi_1(\widetilde K)\supset\{\tfrac{i}{7}:i\in\mathbb Z,\;0\leqslant i\leqslant 7\}$. Finally, the latter together with Lemma \ref{LmChi1Values2} implies that $\chi_1(\widetilde{\mathcal G})\supset \tfrac{1}{7}\mathbb Q_2\cap [0,1]$. Thus,
$\chi_1(\widetilde{\mathcal G})=\tfrac{1}{7}\mathbb Q_2\cap [0,1]=\chi_1(\mathcal G)$.

\section{Embeddings into hyperfinite factor}\label{SecEmbeddings}
Let $\mathfrak R$ be a hyperfinite II$_1$ factor of Murray-von Neumann. Here we show that amenable branch groups have continuously many essentially different embeddings into a unitary group $U(\mathfrak R)$ of $\mathfrak R$.
\begin{Th}\label{ThEmbeddings} Let $G$ be a countable amenable branch group. There is a family $\{\theta_i\}_{i\in I},|I|=2^{\aleph_0}$ of embeddings $G\to U(\mathfrak R)$ belonging to different orbits of the action of $\aut(\mathfrak R)$.
\end{Th}
\begin{proof}
 Let $\chi_C^\p,C\in\mathfrak C_2$, be the family of characters from Theorem \ref{ThMain}, part 2. By Lemma \ref{LmChiCLambda}, $\psi_C=\chi_C^\p$ for each $C$. Set $I=\mathfrak C_2$. We have $|I|=2^{\aleph_0}$ and each character $\psi_C$ is indecomposable. Let $\pi_C$ be the groupoid representation associated with the system $(G,[C],\lambda_{[C]})$. Recall that $\pi_C$ acts in the Hilbert space of the form $L^2(\mathcal R,\nu_{[C]})$ (see Subsection \ref{SubsecNFandGroupoid}). Following the notations from Subsection \ref{SubsecNFandGroupoid}, denote by $\mathcal M_{\pi_C}$
  the $W^*$-algebra generated by the operators of representation $\pi_C$ and by $\mathcal M_{\mathcal R}$ the Murray-von Neumann (or Krieger) algebra generated by $\mathcal M_{\pi_C}$ and operators of multiplication by functions of the form $f(x,y)=m(x)$, $m(x)\in L^\infty([C],\lambda_{[C]})$.

 By Corollary \ref{CoEmptyInterior=>Indecomposable}
 the action $(G,[C],\lambda_{[C]})$ is perfectly non-free.
By \cite{BencsToth17}, this action is ergodic.
Let $\xi(x,y)=\delta_{x,y}\in L^2(\mathcal R,\nu_{[C]})$, \ie $$\xi(x,y)=\left\{\begin{array}{ll}1,&\text{if }x=y,\\0, &\text{otherwise},\end{array}\right. x,y\in[C].$$ Since condition $1)$ of Proposition \ref{PropEqIffIrred} is satisfied, we conclude that $\xi$ is cyclic vector for $\mathcal M_{\pi_C}$, $C\in\mathfrak C_2$. Moreover, the algebra $\mathcal M_{\pi_C}$ is equal to the algebra $\mathcal M_{\mathcal R}$.

By construction, we have:
$$(\pi_C(g)\xi,\xi)=\lambda_{[C]}(\fix(g))=\psi_C(g),g\in G.$$ The factor $\mathcal M_{\pi_C}=\mathcal M_{\mathcal R}$ is a type II$_1$ factor with trace $\tr$ given by $\tr(m)=(m\xi,\xi),\;m\in\mathcal M_{\mathcal R}$.

Since $G$ is amenable, by the result of A. Connes \cite{Connes:1976}, Corollary  6.9, $\mathcal M_{\pi_C}=\mathcal M_{\mathcal R}$ is amenable (equivalently, hyperfinite) finite $W^*$-algebra, and hence is isomorphic to the hyperfinite factor $\mathfrak R$. Let $\tilde\theta_C:\mathcal M_{\pi_C}\to\mathfrak R$ be the isomorphism. Set  $\theta_C(g)=\tilde\theta_C(\pi_C(g))\in U(\mathfrak G)$, $g\in G$. Then $\theta_C$ is an embedding of $G$ into $U(\mathfrak G)$. Let $\tau$ be the canonical trace on $\mathfrak R$. Then $\tau(\theta_C(g))=\chi_C^\p(g),g\in G$. Recall that indecomposable characters on a group $G$ are in a canonical bijection with classes of quasi-equivalence of finite type factor representations of $G$. Since $\chi_C^\p$, $C\in\mathfrak C_2$, are pairwise distinct indecomposable characters on $G$, we obtain that $\aut(\mathfrak R)$ orbits of $\theta_C(G)$, $C\in\mathfrak C_2$, are pairwise distinct, which finishes the proof.
\end{proof}

\section{Appendix: proof of Proposition \ref{PropEqIffIrred}}
Here we follow the notations of Section \ref{SubsecNFandGroupoid}.
\vskip 0.3cm \noindent
$1)\Rightarrow 2)$. Assume that $\mathcal M_{\pi}=\mathcal M_{\mathcal{R}}$. Then by symmetry
$\mathcal M_{\tilde\pi}=\mathcal M_{\mathcal{\tilde{R}}}$. Notice that the commutant of $\rho$
is a subset of $$\mathcal M_{\pi}'\cap\mathcal M_{\tilde\pi}'=
\mathcal M_{\mathcal{R}}'\cap \mathcal M_{\mathcal{\tilde R}}'=\mathcal M_{\mathcal{R}}'\cap\mathcal M_{\mathcal{R}}=\mathbb{C}\id,$$ where
$\id$ is the identity operator and $\mathbb{C}\id$ is the set
of scalar operators in $L^2(\mathcal{R},\nu)$. Therefore,
$\rho$ is irreducible.
\vskip 0.3cm
\noindent
$2)\Rightarrow 3)$. Assume that $\rho$ is irreducible. Then $\xi$ is cyclic with respect to the
algebra generated by $\rho$. Since
$\rho((g,h))\xi=\pi(gh^{-1})\xi$ for all $g,h$ it follows that $\xi$ is cyclic with respect to
$\mathcal{M}_{\pi}$.\vskip 0.3cm
\noindent
$3)\Rightarrow 1)$. Since $\mu$ is invariant with respect to $G$, the modular operator and the modular
automorphism group corresponding to the trace $\tr$ on $\mathcal{M}_{\mathcal{R}}$
are trivial (see \eg \cite{FM2}, Proposition 2.8). By Theorem 4.2 from \cite{Tak2} there
exists a \emph{conditional expectation} $\mathcal{E}:\mathcal{M}_{\mathcal{R}}\to\mathcal{M}_{\pi}$,
that is a linear map such that
\begin{itemize}
\item[$1)$] $\|\mathcal{E}(x)\|\leqslant \|x\|$ and $\mathcal{E}(x)^{*}
\mathcal{E}(x)\leqslant\mathcal{E}(x^*x)$ for all $x\in\mathcal{M}_{\mathcal{R}}$;
\item[$2)$] $\mathcal{E}(x)=x$ for all $x\in\mathcal{M}_{\pi}$;
\item[$3)$] $\tr\circ\mathcal{E}=\tr$;
\item[$4)$] $\mathcal{E}(axb)=a\mathcal{E}(x)b$ for all $a,b\in\mathcal{M}_{\pi}$ and $x\in\mathcal{M}_{\mathcal{R}}$.
\end{itemize}
It follows that for all $x\in\mathcal{M}_{\mathcal{R}}$ one has
$$ \|\mathcal{E}(x)\xi\|^2=\tr(\mathcal{E}(x)^{*}
\mathcal{E}(x))\leqslant\tr(\mathcal{E}(x^*x))= \tr(x^*x)=\|x\xi\|^2.$$ This implies that
the map $$\mathcal{M}_{\mathcal{R}}\xi\to \mathcal{M}_{\pi}\xi,\;\;x\xi\to\mathcal{E}(x)\xi, \;\;x\in \mathcal{M}_{\mathcal{R}}$$ is well defined and extends to  a bounded linear operator $E$ on $L^2(\mathcal{R},\nu)$ of norm $\|E\|\leqslant 1$. Moreover, $E$ is identical on the cyclic hull of $\xi$ under $\mathcal{M}_{\pi}$. Therefore, if $\xi$ is cyclic with
respect to $\mathcal{M}_{\pi}$, then $E=\id$, $\mathcal{E}(x)\xi=x\xi$ and thus
$\mathcal{E}(x)=x$ for all $x\in \mathcal{M}_{\mathcal{R}}$. This implies that $\mathcal{M}_{\pi}=\mathcal{M}_{\mathcal{R}}$.

\end{document}